\documentclass[article,12pt]{amsart}

\usepackage[notref,notcite]{showkeys}%for lectures
\usepackage{amsfonts,url}
\usepackage{color}
\usepackage[margin=23mm]{geometry}
\usepackage{amssymb, amsmath, amsthm, amscd, accents}
\usepackage{graphicx,times,bbm,url}
\usepackage[T1]{fontenc}
\usepackage{eucal,mathrsfs,dsfont}%gives nicer set-names. Use \mathds instead of \mathds

\newcommand{\R}{\mathbb{R}}

\newcommand{\Rd}{ \mathbb{R}^{d}}
\newcommand{\Rdwithouto}{\mathbb{R}^{d}\setminus\{0\}}
\newcommand{\N}{\mathbb{N}}

\newcommand{\indyk}[1]{\mathds{1}_{#1}}

\newcommand{\nuzero}{\accentset{\circ}{\nu}}
\newcommand{\Pzero}{\accentset{\circ}{P}}
\newcommand{\pzero}{\accentset{\circ}{p}}
\newcommand{\sfera}{ \mathds{S}^{d-1}}
\newcommand{\Borel}{ {\mathcal{B}}(\Rd) }
\newcommand{\Fourier}{ {\mathcal{F}}}

\newcommand{\scalp}[2]{#1\cdot#2}

\newtheorem{lemma}{Lemma}
\newtheorem{prop}{Proposition}
\newtheorem{theorem}{Theorem}
\newtheorem{corollary}{Corollary}

\newtheorem{remark}{Remark}
\newtheorem{example}{Example}

\newcounter{conum} 

\renewcommand{\Re}{\ensuremath{\operatorname{Re}}}

\begin{document}

\title{Spatial asymptotics at infinity for heat kernels of integro-differential operators}
\author{Kamil Kaleta and Pawe{\l} Sztonyk}

\address{Kamil Kaleta \\ Faculty of Pure and Applied Mathematics,
  Wroc{\l}aw University of Science and Technology,
  Wybrze{\.z}e Wyspia{\'n}\-skie\-go 27,
  50-370 Wroc{\l}aw, Poland \\  
	and Institut f\"ur mathematische Stochastik,
  Fachrichtung Mathematik,
  Technische Universit\"at Dresden,
  01062 Dresden, Germany}
\email{kamil.kaleta@pwr.edu.pl}

\address{Pawe{\l} Sztonyk \\ Faculty of Pure and Applied Mathematics,
  Wroc{\l}aw University of Science and Technology,
  Wybrze{\.z}e Wyspia{\'n}\-skie\-go 27,
  50-370 Wroc{\l}aw, Poland}
\email{pawel.sztonyk@pwr.edu.pl}

\maketitle

\begin{abstract}
We study a spatial asymptotic behaviour at infinity of kernels $p_t(x)$ for convolution semigroups of nonlocal pseudo-differential operators.
We give general and sharp sufficient conditions under which the limits
$$
  \lim_{r \to \infty} \frac{p_t(r\theta-y)}{t \, \nu(r\theta)}, \quad t \in T, \ \ \theta \in E, \ \ y \in \R^d,
$$
exist and can be effectively computed. Here $\nu$ is the corresponding L\'evy density, $T \subset (0,\infty)$ is a bounded time-set and $E$ is a subset of the unit sphere in $\R^d$, $d \geq 1$. Our results are local on the unit sphere. They apply to a wide class of convolution semigroups, including those corresponding to highly asymmetric (finite and infinite) L\'evy measures. Key examples include fairly general families of stable, tempered stable, jump-diffusion and compound Poisson semigroups. A main emphasis is put on the semigroups with L\'evy measures that are exponentially localized at infinity, for which our assumptions and results are strongly related to the existence of the multidimensional exponential moments. Here a key example is the evolution semigroup corresponding to the so-called \emph{quasi-relativistic Hamiltonian} $\sqrt{-\Delta+m^2} - m$, $m>0$. As a byproduct, we also obtain sharp two-sided estimates of the kernels $p_t$ in generalized cones, away from the origin.


\bigskip
\noindent
2010 {\it MS Classification}: Primary 47D03, 60J35, 35A08; Secondary 60G51, 60E07, 35S10.\\
%47D03 Groups and semigroups of linear operators
%60J35 Transition functions, generators and resolvents 
%60G51 Processes with independent increments
%47G30 	Pseudodifferential operators
%35S10   	Initial value problems for pseudodifferential operators
%35A08   	Fundamental solutions
  
\end{abstract}

\footnotetext{
Research was supported in part by the National Science Centre, Poland, grants no. 2015/17/B/ST1/01233 and 2016/23/G/ST1/04211 and by the Alexander von Humboldt Foundation, Germany.
}
%----------------------------------------------------------------
\section{Introduction and statement of results}

In recent years, nonlocal integro-differential operators and the corresponding evolution equations have 
received much attention in both pure and applied mathematics. Nonlocal operators and related 
stochastic processes, often called diffusions with jumps, provide new methods in scientific modelling, 
in particular they allow us to model discontinuous phenomena, providing realistic correctives and refinements 
to established theories.

Let $d \geq 1$, $b\in\Rd$, $A = (a_{ij})_{1 \leq i,j \leq d}$ be a symmetric nonnegative definite matrix, and let $\nu$ be a measure on $\R^d \setminus \left\{ 0 \right\}$ such that $\int_{\Rdwithouto} \left(1\wedge |y|^2\right)\,\nu(dy) < \infty$, called L\'evy measure. In this paper, under fairly general conditions on $\nu$ and $A$, we study a spatial asymptotic behaviour at infinity of the fundamental solution $p_t(x) := p(t,x)$ (the heat kernel) to the following nonlocal evolution equation
$$
\partial_t u(t,x) - L_x u(t,x) = 0, \quad x \in \R^d, \ \ t >0,
$$  
where $L$ is a homogeneous pseudo-differential operator which is uniquely determined by its Fourier transform
$$
\Fourier(L \, h)(\xi) = -\psi(\xi) \Fourier(h)(\xi), \quad \xi \in \R^d, \ \ h \in D(L) := \left\{g \in L^2(\R^d): \psi \Fourier(g) \in L^2(\R^d) \right\},
$$
where
\begin{equation} \label{eq:Phi_2}
  \psi(\xi) =  - i\scalp{\xi}{b} + \scalp{\xi}{A\xi}  + \int \left(1 - e^{i\scalp{\xi}{y}} + i\scalp{\xi}{y}\indyk{B(0,1)}(y)\right)\nu(dy) , \quad \xi\in\Rd.
\end{equation}
It is known that $C_c^{\infty}(\R^d) \subset D(L)$ is a core of $L$ on which it has the following integro-differential representation 
$$
L \, h(x) = \scalp{b}{\nabla h (x)} + \sum_{i,j=1}^d a_{ij} \partial_{x_i}\partial_{x_j} h(x)
  + \int \left(h(x+z)-h(x)-\indyk{B(0,1)}(z)\scalp{z}{\nabla h(x)}\right) \,\nu(dz).
$$
The operator $L$ is a generator of the L\'evy process with jumps which is fully described by a convolution semigroup of probability measures $\{ P_t,\, t\geq 0 \}$ on $\R^d$ such that $\Fourier(P_t)(\xi)=\int_{\Rd} e^{i\scalp{\xi}{y}}P_t(dy)=\exp(-t\psi(\xi))$, $\xi \in \R^d$, $t>0$. More precisely, its transition probabilities have the form $P_t(B-x)$, $x \in \R^d$, $B \in \Borel$. The functions $p_t(\cdot)$, whenever they exist, are densities of measures $P_t$. For regular introduction to the theory of pseudo-differential operators, their evolution semigroups and related L\'evy and L\'evy-type processes we refer to \cite{BSW, J1}. The existence and the regularity of densities for convolution semigroups are discussed in \cite{KSch} (see also \cite{JKLSch2012})

The explicit expression of $p_t(x)$ is typically impossible to get. Therefore, it is a basic problem, both in probability theory and in analysis, to obtain the estimates as well as some information on the asymptotic behaviour of $p_t(x)$ in space and time. In case of uniformly elliptic and bounded divergence form operators, which generate the diffusion processes in $\R^d$, it is well known that the heat kernels enjoy the celebrated Aronson's Gaussian type behaviour \cite{A}. 

Investigations on asymptotic behaviour of isotropic $\alpha$-stable ($\alpha \in (0,2)$) convolution semigroups date back to 1923 and 1960, when P\'olya \cite{P} and Blumenthal and Getoor \cite{BG60} obtained the first results in this direction. With respect to a further study of asymptotic behaviour of convolution semigroups in space and time we refer to \cite{B, Ishi94, Lean, Y, RW, FH, KnopKul, KT} and references there. In recent papers \cite{CGT, KM, GRT} the case of unimodal and isotropic jump L\'evy processes has been analyzed.  

The paper which is the most related to our present work is the well known contribution of J. Dziuba\'nski \cite{D91}, where similar asymptotic problem for strictly stable semigroups on Lie groups, including Euclidean spaces, was studied. The argument in this paper is based on perturbation techniques and scaling properties, and essentially differs from our approach (see further discussion in Section 6.1). Our methods allows us to deal with a fairly general class of homogeneous integro-differential operators $L$ in Euclidean spaces, under reasonable conditions on $A$ and $\nu$. We do not require any scaling conditions and include the operators with highly anisotropic integral parts, with finite and infinite L\'evy measures. Our argument is mainly based on a precise analysis of the radial asymptotics at inifnity for densities of convolutions of restricted multidimensional L\'evy measures $\nu_r(\,\cdot \,) := \nu(\, \cdot \, \cap B(0,r)^c)$ for large $r>0$, and for the corresponding convolution exponents, which form a certain family of compound Poisson semigroups of measures. This can be effectively done under very powerful assumption involving the particular parameter function $K$ (see \eqref{ass:sjp}). It provides us with a sufficient control of single convolutions $\nu_r * \nu_r(x)$ in large $x$ and $r$ and gives a necessary compactification of convergence. These ideas are completely new in the context of asymptotic behaviour of convolution semigroups.
A remarkable feature of our study is that we cover not only long-tailed L\'evy measures (like that of a jump-stable type L\'evy process), but also those with second moment finite, including exponentially localized L\'evy measures, which turn out to be the most difficult case. Neither results nor methods of this type were previously known in this case. For instance, we derive the spatial asymptotics at infinity for heat kernels of relativistic stable operators (Section 6.2), tempered stable (Section 6.4) and compound Poisson semigroups, which are related to the so-called convolution operators. These classes of operators and corresponding L\'evy processes are known to have interesting and important applications in (mathematical) physics and technical sciences \cite{CMS, LS, MS, N, KL16, Kop, N, UM}, financial methematics \cite{CGMY, CGMY2, BRKF, Ko, Sc} or even atmospheric sciences \cite{Ka}.

We now turn to the presentation of our results. Denote 
$$
\Phi(\xi) = \int \left(1 - e^{i\scalp{\xi}{y}} + i\scalp{\xi}{y}\indyk{B(0,1)}(y)\right)\nu(dy), \quad \xi \in R^d,
$$
and 
$$
\Psi(r)= (\Re \Phi)^{*}(r) := \sup_{|\xi|\leq r} \Re \Phi(\xi), \quad r>0.
$$
Clearly, $\Phi(\xi)$ is a part of the Fourier symbol $\psi(\xi)$ which corresponds to the integral part of the operator $L$ and $\Psi(r)$ is a maximal function of its symmetrization. We note that $\Psi$ is continuous, non-decreasing and $\Psi(0)=0$. Denote $\Psi(\infty):= \lim_{r \to \infty} \Psi(r) = \sup_{r>0} \Psi(r)$. One can check that $\Psi(\infty) = \infty$ if $\nu(\Rdwithouto)=\infty$.
Also, let 
$$
\Psi_{-}(s)=\sup\{r>0: \Psi(r)=s\}, \quad  s\in \big(0,\Psi(\infty)\big),
$$ 
be the generalized right inverse function to $\Psi$. We have $\Psi(\Psi_{-}(s))=s$ for $s\in \big(0,\Psi(\infty)\big)$ and $\Psi_{-}(\Psi(s))\geq s$ for $s>0$. For $E \subset \sfera$ we denote $\Gamma_{E}:= \left\{y:y/|y|\in E\right\}$.

\medskip

The following will be the standing assumptions on the Gaussian matrix $A$ and the L\'evy measure $\nu$ throughout the paper. 

\bigskip

\begin{itemize}
\item[\textbf{(A)}]  
$A \equiv 0$ \ or \ $\inf_{|\xi|=1} \xi \cdot A \xi > 0 $
\end{itemize} 

\bigskip

\begin{itemize}
\item[\textbf{(B)}]
$\nu(dx) = \nu(x) dx$ and there exists a nonincreasing function $f:(0,\infty) \to (0,\infty)$ and a constant $C_0 >0$ such that
$$
\nu(x) \leq C_0 f(|x|), \quad x \in \Rdwithouto, 
$$
\begin{align} \label{ass:low_reg}
\quad \liminf_{r \to 0^{+}} \frac{\nu\big( \left\{x: |x| > r\right\}\big)}{f(r)r^d} >0
\end{align}
and 
\begin{align} \label{ass:sjp}
K(r):= \sup_{|x| > 1} \frac{\int_{|x-y|>r \atop |y| >r} f(|x-y|) f(|y|) dy}{f(|x|)} \ \searrow \ 0 \qquad \text{as \ \ $r \to \infty$.}
\end{align}
\end{itemize}

\bigskip

\begin{itemize}
\item[\textbf{(C)}]  
Let $E \subset \sfera$ and $\kappa \geq 0$ be such that
\begin{align} \label{eq:add_conv}
\lim_{r \to \infty} \frac{\nu(r \theta-y)}{\nu(r\theta)} = e^{\kappa (\scalp{\theta}{y})}, \quad y \in \R^d, \ \ \theta \in E,
\end{align}
and 
\begin{align} \label{eq:lower}
C_1 := \inf_{x \in \Gamma_E} \frac{\nu(x)}{f(|x|)}   > 0.
\end{align}
\end{itemize} 

\bigskip

\begin{itemize}
\item[\textbf{(D)}]  
There is a nonempty and bounded set $T \subset (0,\infty)$ and a constant $C_2 >0$ such that
$$
\int_{\R^d} e^{-t \Re \Phi(\xi)} |\xi| d\xi \leq C_2 \left(\Psi_{-}\left(\frac{1}{t}\right)\right)^{d+1}, \quad t \in T.
$$
\end{itemize} 

\bigskip

\noindent
The first two assumptions \textbf{(A)} and \textbf{(B)} give a general framework for our study and determine the class of semigroups we work with. 
The condition \textbf{(C)} determines the type of convergence which is initially required from the densities of the underlying L\'evy measures.
Both conditions \textbf{(B)} and \textbf{(C)} are fundamental for the results obtained in the present paper. The last condition \textbf{(D)} is rather a technical assumption, which provides the existence and required regularity of the densities $p_t(x)$ for the measures $P_t(dx)$ over the given time-set $T$. Further discussion of the assumptions \textbf{(A)}-\textbf{(D)} is
given in Remark \ref{rem:assumptions}.

By \cite[Th. 25.17]{Sato}, for given $\xi \in \R^d$, the condition
\begin{align} \label{eq:v-moment}
\int_{|y| > 1} e^{\scalp{\xi}{y}} \nu(dy) < \infty
\end{align}
is equivalent to the existence of multidimensional exponential moment of order $\xi$ of the semigroup $\{ P_t,\, t\geq 0 \}$, i.e., 
$$
\int_{\R^d}  e^{\scalp{\xi}{y}} P_t(dy) < \infty, \quad t> 0. 
$$
Moreover, the function 
\begin{align} \label{eq:hyp_exp}
\widetilde \psi(\xi) = - \scalp{\xi}{b} - \scalp{\xi}{A \xi} + \int_{\R^d \setminus \left\{0\right\}} 
		\left(1-e^{\scalp{\xi}{y}}+\scalp{\xi}{y}\indyk{B(0,1)}(y) \right) \nu(y)\, dy
\end{align}
is definable and finite for every $\xi \in \R^d$ satisfying \eqref{eq:v-moment} and the equality
\begin{align} \label{eq:hyp_exp_1}
\int_{\R^d}  e^{\scalp{\xi}{y}} P_t(dy) = e^{-t \widetilde \psi(\xi)}, \quad t >0,
\end{align}
holds. 
%As we will see from Lemma \ref{lm:useful_2} below, the assumptions \textbf{(B)} and \textbf{(C)} give that
%\eqref{eq:v-moment} holds with $\xi = \kappa \theta$, for every $\theta \in E$. In particular, the function 
%$\widetilde \psi(\kappa \theta)$ is well defined for $\theta \in E$. 
We prove below in Lemma \ref{lm:useful_2} that our assumptions \textbf{(B)} and \textbf{(C)} yield \eqref{eq:v-moment} for every $\xi=k\theta$, with $\theta\in E$. In particular, the map $\theta \mapsto \widetilde \psi(\kappa \theta)$ is well defined and uniformly bounded on $E$ and \eqref{eq:hyp_exp_1} holds for all $\xi=\kappa\theta$, $\theta \in E$. 

\bigskip

The following theorem is a main result of this paper.

\begin{theorem} \label{thm:main}
Let the assumptions \textbf{(A)}-\textbf{(D)} hold. Specifically, let \textbf{(C)} and \textbf{(D)} be satisfied with some $E \subset \sfera$, $\kappa \geq 0$ and $T \subset (0, \infty)$. 
Then the following hold.
\begin{itemize}
\item[(a)] For every $t \in T$, $\theta \in E$ and $y \in \R^d$, 
\begin{align} \label{eq:res_conv}
  \lim_{r \to \infty} \frac{p_t(r\theta -y)}{t \, \nu(r\theta)} = \left\{
  \begin{array}{ccc}
   1 & \mbox{  if  } & \kappa = 0,\\
   e^{-t \widetilde \psi(\kappa \theta) + \kappa (\scalp{\theta}{y})} & \mbox{  if  } & \kappa > 0.
  \end{array}\right.
\end{align}
\item[(b)] If for every compact set $D \subset \R^d$ the convergence in \textbf{(C)} is uniform in $(\theta,y)$ on the rectangle $E \times D$, then \eqref{eq:res_conv} is uniform in $(t,\theta,y)$ on each cuboid $T \times E \times B(0,\varrho) $, $\varrho >0$.
\end{itemize}
\end{theorem}

\bigskip

We now discuss in a more detail our assumptions and results.

\noindent
\begin{remark} \label{rem:assumptions}
\rm{

\noindent
\begin{itemize}
\item[(a)] The assumption \textbf{(A)} is self-explanatory (cf. \cite{A}). Observe that the differential part of the operator $L$ induced by the matrix $A$ does not play any important role in the main-order term in the asymptotics of the heat kernel $p_t(x)$. It only contributes to the exponent $\widetilde \psi(\kappa \theta)$ if $\kappa >0$, i.e. if the decay of $\nu$ at infinity is exponential.

\item[(b)] We like to note that the density $\nu$ is not required to have any radiality, symmetry and monotonicity properties. We only assume in \textbf{(B)} that it is dominated by a nonincreasing profile $f$ having certain regularity properties \eqref{ass:low_reg}-\eqref{ass:sjp} and agreeing with $\nu$ on a given generalized cone only (as in \eqref{eq:lower}). The property \eqref{ass:low_reg} means that on average the profile $f$ is sharp enough to reflect the behaviour of $\nu$ at zero. The function $K$ appearing in \eqref{ass:sjp} has been recently introduced in \cite[Sec. 2.2]{KL17} as a parameter describing the long jumps properties of L\'evy processes driven by measures $P_t$. It has an interesting stochastic interpretation: if $\nu \asymp f$, then $K(r)$ represents the rate of preference of single jumps of size
at least $r$ over arbitrary combinations of double jumps of size at least $r$. It is clearly a nonincreasing function such that $K(r) \geq C \nu(B(0,r)^c)$, $r \geq 1$ \cite[Lem. 2.1]{KL17}. In the cited paper, $K$ was a main tool in a study of the localization properties of eigenfunctions of nonlocal Schr\"odinger operators corresponding to negative eigenvalues. Note that if $\nu$ is a radial nonincreasing function (in this case all $p_t$ inherit these properties and the corresponding semigroup is called isotropic unimodal), our result in Theorem \ref{thm:main} is sharp as well.

\item[(c)] The condition \textbf{(C)} is local in the sense that it may be satisfied only for subsets $E$ of the unit sphere (including singletons), leading to asymptotic results for $p_t$ on these sets (cf. Example \ref{ex:ex1}). In particular, from the assertion (b) of Theorem \ref{thm:main} we can easily derive that if \textbf{(C)} holds with $E=\{\theta\}$ for given $\theta \in\sfera$ and the convergence in \eqref{eq:add_conv} is uniform in $y$ on every compact set $D\subset\Rd$, we get uniform convergence in \eqref{eq:res_conv} in $(t,y)$ on $T\times B(0,\varrho)$ for every $\varrho>0$. Our condition \eqref{eq:add_conv} can be seen as a one of possible multidimensional generalizations of the analogous asymptotic property known from the theory of one-dimensional sub-exponential and convolution-equivalent distributions in probability. With respect to a study in this area and some applications we refer to \cite{Klup, Pa1, Pa2, W05, W10, Knop, R, T} and references there, just to mention a few contributions.

\item[(d)] It is instructive to see how essential for our results are the conditions \eqref{ass:sjp} and \eqref{eq:add_conv}. Some possible converse implications between the convergence \eqref{eq:res_conv} in Theorem \ref{thm:main} and these conditions are discussed in Proposition \ref{prop:converse} in Section 4. It is also worth to point out that the convergence in \eqref{eq:add_conv} is not enough for the existence of exponential moments of $p_t$, and, in consequence, for the convergence \eqref{eq:res_conv} in Theorem \ref{thm:main}. Here the control of the second convolution as in \eqref{ass:sjp} is crucial as well (cf. Example \ref{ex:ex2}). Moreover, it can be conjectured that the condition $K(1) = \sup_{r \geq 1} K(r) < \infty$ is actually not very far from the assumption that $K(r) \to 0$ as $r \to \infty$ in \eqref{ass:sjp}. For some other applications of the condition $K(1)<\infty$ and further discussion of it we refer the reader to our recent papers \cite{KL, KSz, KSz2, KL16, KKL17}.

\item[(e)] As noticed in Section 6.3, the inequality in \textbf{(D)} depends only on the behaviur of $\nu$ around zero, which translates to the behaviour of $\Re\Phi$ at infinity. Observe that if $\int e^{-t_0 \Re \Phi(\xi)} |\xi| d\xi < \infty$ for some $t_0>0$, then, thanks to the monotonicity, \textbf{(D)} holds true for every $T = [t_0, t_1]$ with $t_1 > t_0$. It is easy to check that this integrability follows e.g. from the Hartman-Wintner type condition $\liminf_{|\xi| \to \infty} \frac{\Re \Phi(\xi)}{\log|\xi|} > 0$. On the other hand, one can verify that if $\nu \asymp f$ and there exists $\alpha >0$, $r_0 >0$ and $C \in (0,1)$ such that $\Re \Phi(\lambda \xi) \geq C \lambda^{\alpha} \Re \Phi(\xi)$, for every $|\xi| > r_0$ and $\lambda \geq 1$, then there exists $t_0>0$ such that the assumption \textbf{(D)} holds with $T = (0,t_0)$ (see e.g. \cite[Lem. 5]{KSz2} and \cite{KSz1}). Some examples are discussed in Section \ref{sec:Examples}.  
\end{itemize}
}
\end{remark}
 
The following two-sided sharp estimate of $p_t$ in generalized cones $\Gamma_{E}$, away from the origin, is a direct corollary from Theorem \ref{thm:main}. It can be seen as a spherically local version of our estimates in \cite[Thm. 3 and Thm. 4]{KSz2}.

\begin{corollary} \label{cor:main}
Let the assumptions \textbf{(A)}-\textbf{(D)} hold. Specifically, let \textbf{(C)} and \textbf{(D)} be satisfied with some $E \subset \sfera$, $\kappa \geq 0$ and $T \subset (0, \infty)$. If for every compact set $D \subset \R^d$ the convergence in \textbf{(C)} is uniform in $(\theta,y)$ on the rectangle $E \times D$, then for every $\varrho >0$ and $\varepsilon \in (0,1)$ there exists $R >0$ such that 
$$
  \left(e^{-t \widetilde \psi(\kappa \theta) + \kappa (\scalp{\theta}{y})}-\varepsilon\right) \, t \, \nu(x) \leq p_t(x -y) \leq \left(e^{-t \widetilde \psi(\kappa \theta) + \kappa (\scalp{\theta}{y})}+\varepsilon\right) \, t \, \nu(x),
$$
for every $t \in T$, $y \in B(0,\varrho)$ and $x \in \Gamma_E \cap B(0,R)^c$.
\end{corollary}

The above bounds are of special interest if $T \supseteq (0,t_0)$, for some $t_0 >0$. As proven in Lemma \ref{lm:useful_2} below, the function $\theta \mapsto \widetilde \psi(\kappa \theta)$ is uniformly bounded on $E$, which gives that $0 < e^{-\sup T |\widetilde \psi(\kappa \theta)|-\kappa \varrho} \leq e^{-t \widetilde \psi(\kappa \theta) +\kappa (\scalp{\theta}{y})}$, for every $t \in T$, $\theta \in E$ and $y \in B(0,\varrho)$. 

Our second theorem is devoted to finite L\'evy measures. In this case, the condition \textbf{(D)} can not hold for any nonempty set $T \subset (0,\infty)$. If $\inf_{|\xi|=1} \xi \cdot A \xi > 0$ in \textbf{(A)}, then despite the fact that $\nu(\R^d \setminus \left\{0\right\}) < \infty$ each measure $P_t(dx)$ is absolutely continuous with respect to the Lebesgue measure with bounded density $p_t(x)$. L\'evy processes driven by this type of convolution semigroups are often called \emph{jump-diffusions} and play an important role in scientific modelling (see e.g. \cite{Ko}). On the other hand, when $A \equiv 0$ and $\nu(\R^d \setminus \left\{0\right\}) < \infty$, then $\{ P_t,\, t\geq 0 \}$ is a compound Poisson semigroup of measures on $\R^d$, possibly with drift and atoms. In this case, we can still examine the spatial asymptotics at infinity of the functions 
$$\widetilde{p}_t(x) := e^{-t|\nu|} \sum_{n=1}^\infty \frac{t^n\nu^{n*}(x)}{n!},$$ 
which are densities of the absolutely continuous components of $P_t$ (for more details see Preliminaries).

\begin{theorem} \label{thm:main_second}
Let the assumptions \textbf{(A)}-\textbf{(C)} hold. Specifically, let \textbf{(C)} hold with some $E \subset \sfera$ and $\kappa \geq 0$ 
and suppose that $\nu(\R^d \setminus \left\{0\right\}) < \infty$.
If $\inf_{|\xi|=1} \xi \cdot A \xi > 0$ in \textbf{(A)}, then the following hold.
\begin{itemize}
\item[(a)] For every $t > 0$, $\theta \in E$ and $y \in \R^d$, \eqref{eq:res_conv} holds true. 
\item[(b)] If for every compact set $D \subset \R^d$ the convergence in \textbf{(C)} is uniform in $(\theta,y)$ on the rectangle $E \times D$, then \eqref{eq:res_conv} is uniform in $(t,\theta,y)$ on each cuboid $(0,t_0] \times E \times B(0,\varrho) $, $\varrho >0$, $t_0>0$.
\end{itemize}
On the other hand, if $A \equiv 0$, then the same statements (a)-(b) hold for $p_t$ replaced with $\widetilde{p}_t$ in \eqref{eq:res_conv}.
\end{theorem}

The analog of Corollary \ref{cor:main} for finite L\'evy measures resulting from Theorem \ref{thm:main_second} holds as well.

\begin{corollary} \label{cor:main_2}
Let the assumptions \textbf{(A)}-\textbf{(C)} hold. Specifically, let \textbf{(C)} be satisfied with some $E \subset \sfera$ and $\kappa \geq 0$, and suppose that $\nu(\R^d \setminus \left\{0\right\}) < \infty$. If $\inf_{|\xi|=1} \xi \cdot A \xi > 0$ in \textbf{(A)} and for every compact set $D \subset \R^d$ the convergence in \textbf{(C)} is uniform in $(\theta,y)$ on the rectangle $E \times D$, then for every $\varrho >0$ and $\varepsilon \in (0,1)$ there exists $R >0$ such that 
$$
\left(e^{-t \widetilde \psi(\kappa \theta) + \kappa (\scalp{\theta}{y})}-\varepsilon\right) \, t \, \nu(x) \leq p_t(x -y) \leq \left(e^{-t \widetilde \psi(\kappa \theta) + \kappa (\scalp{\theta}{y})}+\varepsilon\right) \, t \, \nu(x),
$$
for every $t \in T$, $y \in B(0,\varrho)$ and $x \in \Gamma_E \cap B(0,R)^c$. If $A \equiv 0$, then the same bounds hold for $\widetilde{p}_t$.
\end{corollary}
We like to emphasize that compound Poisson semigroups and related L\'evy processes are also widely used in practice, mainly in queuing and risk theory (see e.g. \cite{EKM, Sc, Ka} and references therein). 

The rest of the paper is organized as follows. In Section 2 we recall the standard facts on decompositions of convolution semigroups and prove some auxiliary results. Some of them are of independent interest. In Section 3 we establish the asymptotics at infinity for densities of convolutions of the restricted L\'evy measures $\nu_r$ and for the densities $\bar{p}_t$ of the absolutely continuous parts of the corresponding convolution exponents with $r= 1/\Psi_{-}(1/t)$. Section 4 contains the proof of Theorem \ref{thm:main}, our main result, and Proposition \ref{prop:converse}. In Section 5 we first collect several auxiliary results for finite L\'evy measures which are counterparts of those in Section 2 and then we apply them to prove our second main result, Theorem \ref{thm:main_second}. Section 6 is devoted to detail discussion of applications of our general results to several particular classes of convolution semigroups.   

\bigskip

\section{Preliminaries}

\noindent
Throughout the paper we assume that $\{ P_t,\, t\geq 0 \}$ is a convolution semigroup of probability measures on $\R^d$, $d\in \left\{1,2,...\right\}$, which is uniquely determined by \eqref{eq:Phi_2} with an arbitrary $b \in \R^d$, and a Gaussian matrix $A$ and a L\'evy measure $\nu$ satisfying our framework assumptions \textbf{(A)} and \textbf{(B)}. 

For every $r>0$ we denote by $\{\Pzero^r_t,\; t\geq 0\}$ and $\{\bar{P}^r_t,\; t\geq 0\}$ the semigroups of measures determined by 
$$
  \Fourier(\Pzero^r_t)(\xi) 
    =     \exp\left(t \int_{\Rdwithouto} \left(e^{i\scalp{\xi}{y}}-1-i\scalp{\xi}{y}\right)
            \mathring{\nu}_r(y)dy\right)\, ,  \quad \xi\in\Rd\,, \ t>0 \, ,
$$
and
$$
  \Fourier(\bar{P}^r_t)(\xi) =
  \exp\left(t \int (e^{i\scalp{\xi}{y}}-1)\,
  \nu_r(y)dy\right)\, ,
  \quad \xi\in\Rd\,, \ t>0 \, ,
$$
with
$$
\nuzero_r(y)=\indyk{B(0,r)}(y)\nu(y) \quad \text{and} \quad \nu_r(y)= \indyk{B(0,r)^c}(y)\,\nu(y), \quad r > 0,
$$
respectively. Note that $\{\bar{P}^r_t,\; t\geq 0\}$ is a compound Poisson semigroup of probability measures of the form
$$
  \bar{P}^r_t(dx) =  \exp(t(\nu_r-|\nu_r|\delta_0))(dx) = e^{-t|\nu_{r}|}\delta_{0}(dx) + {\bar{p}}^r_t(x) dx \,,\quad t > 0\,, \ r >0 \, ,
$$
with
$$
\bar{p}^r_t(x) := e^{-t|\nu_{r}|} \sum_{n=1}^\infty \frac{t^n\nu_r^{n*}(x)}{n!},
$$
where $\nu_r^{n*}(x)$ denotes the densities of the $n$-fold convolutions $\nu_r^{n*}(dx)$ of the finite L\'evy measures $\nu_r(dx)= \nu_r(x)dx$.
Furthermore, since
\begin{eqnarray*}
  |\Fourier(\Pzero^r_t)(\xi)| 
  &   =  & \exp\left(-t\int_{0<|y|<r}
           (1-\cos(\scalp{y}{\xi}))\,
           \nu(y)dy\right) \nonumber \\
  &   =  & \exp\left(-t\left(\Re \Phi(\xi)-\int_{|y|\geq r}
           (1-\cos(\scalp{y}{\xi}))\,
           \nu(y)dy\right)\right) \nonumber \\
  & \leq & \exp(-t\Re \Phi(\xi))\exp(2t\nu(B(0,r)^c)),
           \quad \xi\in\Rd, \ r>0, \ t > 0,
\end{eqnarray*}
under the assumption \textbf{(D)}, for every $r>0$ and $t \in T$, the measures $\Pzero^r_t$ are absolutely continuous with respect to
the Lebesgue measure with densities $\pzero^r_t\in C^1_b(\Rd)$. Also, whenever $A \neq 0$, by $\{G_t,\; t\geq 0\}$ we denote the semigroup of Gaussian measures determined by 
$$
  \Fourier(G_t)(\xi) =
  \exp\left(- t \, \scalp{\xi}{A \xi} \right)\, ,
  \quad \xi\in\Rd\, . 
$$
All $G_t(dx)$ are absolutely continuous with respect to the Lebesgue measure with densities $g_t(x)$. To shorten the notation below, we set 
\begin{equation}\label{eq:def_br}
h(t):= \frac{1}{\Psi_{-}\left(\frac{1}{t}\right)}
\quad  \text{and} \quad 
  b_r := \left\{
  \begin{array}{lcc}
    b - \int_{r \leq |y|<1} y \, \nu(y)dy & \mbox{  if  } & r < 1,\\
		b                                & \mbox{  if  } & r = 1, \\
		b + \int_{1 \leq |y|<r} y \, \nu(y)dy & \mbox{  if  } & r > 1.
  \end{array}\right. 
\end{equation}
With the above notation, under the assumption \textbf{(D)}, for every $r >0$ and $t \in T$ we have the following:
\begin{align} \label{eq:basic_conv_1}
\mbox{if \ $A \neq 0$, \ then \ $P_t=G_t \ast \Pzero^r_t \ast \bar{P}^r_t \ast \delta_{t b_r} $ \ \ 
and \ \ $p_t= e^{-t|\nu_r|} g_t \ast\pzero^r_t \ast \delta_{t b_r} + g_t \ast \pzero^r_t \ast {\bar{p}}^r_t \ast \delta_{t b_r}$,} 
\end{align}
and
\begin{align} \label{eq:basic_conv_2}
\mbox{if \ $A \equiv 0$, \ then \ $P_t=\Pzero^r_t \ast \bar{P}^r_t \ast \delta_{t b_r} $ \ \ 
and \ \ $p_t= e^{-t|\nu_r|} \pzero^r_t \ast \delta_{t b_r} + \pzero^r_t \ast {\bar{p}}^r_t \ast \delta_{t b_r}$.} 
\end{align}
The decomposition formulas \eqref{eq:basic_conv_1}-\eqref{eq:basic_conv_2} will be a starting point in the proof of our main result in Theorem \ref{thm:main}. They will be applied with $r=h(t)$, and, therefore, for simplification, below we will write $\pzero_t=\pzero^{h(t)}_t$ and $\bar{p}_t=\bar{p}^{h(t)}_t$. 

In the proof of Theorem \ref{thm:main_second} we will also need a version of \eqref{eq:basic_conv_1}-\eqref{eq:basic_conv_2} for finite L\'evy measures. If $\nu(\R^d \setminus \left\{0\right\}) < \infty$, then by $\{\widetilde{P}_t,\; t\geq 0\}$ we denote a compound Poisson semigroup of probability measures determined by 
$$
  \Fourier(\widetilde{P}_t)(\xi) =
  \exp\left(t \int (e^{i\scalp{\xi}{y}}-1)\,
  \nu(y)dy\right)\, ,
  \quad \xi\in\Rd\,, \ t>0 \, .
$$
Each measure $\widetilde{P}_t$ has the form
$$
  \widetilde{P}_t(dx) =  \exp(t(\nu-|\nu|\delta_0))(dx) = e^{-t|\nu|} \delta_{0}(dx) + {\widetilde{p}}_t(x) dx \,,\quad t\geq 0\, ,
$$
with
$$
\widetilde{p}_t(x) := e^{-t|\nu|} \sum_{n=1}^\infty \frac{t^n\nu^{n*}(x)}{n!}.
$$
Thus, under the assumption $\nu(\R^d \setminus \left\{0\right\}) < \infty$, for every $t>0$, we have the following:
\begin{align} \label{eq:basic_conv_3}
\mbox{if \ $A \neq 0$, \ then \ $P_t= G_t \ast \widetilde{P}_t \ast \delta_{t \widetilde b} $ \ \ 
and \ \ $p_t= e^{-t|\nu|} g_t \ast \delta_{t \widetilde b} 
            + g_t \ast  \widetilde{p}_t(x) \ast \delta_{t \widetilde b}$,} 
\end{align}
and
\begin{align} \label{eq:basic_conv_4}
\mbox{if \ $A \equiv 0$, \ then \ $P_t= \widetilde{P}_t \ast \delta_{t \widetilde b} = e^{-t|\nu|} \delta_{t \widetilde b}(dx) + {\widetilde{p}}_t(x - t \widetilde b) dx  $,} 
\end{align}
where $\widetilde b = b - \int_{|y|<1} y \nu(y) dy$. In the latter case, each $P_t$ has an atom at $t \widetilde b$. 

Recall that
\begin{align} \label{def:G}
K(r):= \sup_{|x| > 1} \frac{\int_{|x -y| > r \atop |y| > r } f(|x-y|) f(|y|) dy}{f(|x|)}, \qquad r \geq 1,
\end{align}
is the parameter function appearing in the assumption \textbf{(B)}. The direct consequence of this assumption is that $K(1) < \infty$,
which has a remarkable impact on the decay properties of the functions $\nu_r^{\ast n}(x)$ and $\bar{p}_t(x)$ at infinity and provides some extra regularity of the profile function $f$. The following lemma collects some useful and basic auxiliary estimates that are a straightforward consequence of the results obtained recently in \cite{KSz2}. In what follows we will often use the fact that (see e.g. \cite[Proposition 1]{KSz1})  
\begin{equation}\label{eq:Psi_eigenschaften}
  |\nu_r|\leq C_3 \Psi(1/r) \ \quad \  \text{with some \ $C_3>0$ \ for $r>0$ \ \ \  and} \ \quad  \ \sup_{r>0} \frac{\Psi(2r)}{\Psi(r)} < \infty.
\end{equation}
The latter growth control condition is often referred as the doubling property of the function $\Psi$.

\begin{lemma}\label{lm:useful} 
Let the assumption \textbf{(B)} holds. Then for every fixed $r_0>0$ we have the following. 
\begin{itemize}
\item[(a)]
There are constants $C_4=C_4(r_0)$ and $C_5=C_5(r_0)$ such that for every $|x| \geq 2r_0$ and $r \in (0,r_0]$ one has
\begin{align*} %\label{eq:G_1}
\int_{|x -y| > r_0 \atop |y| > r } f(|x-y|) \nu(y) dy \leq C_4 \Psi\left(\frac{1}{r}\right) f(|x|) \quad \text{and} \quad f(r) \leq C_5 \Psi\left(\frac{1}{r}\right)\frac{1}{r^d}.
\end{align*}
\item[(b)]
There exists a constant $C_6=C_6(r_0) \geq 1$ such that
$$
f(s-r_0) \leq C_6 f(s), \quad s \geq 3r_0.
$$
\item[(c)] For every numbers $C_7, C_8 >0$ there exists a constant $C_9:=C_9(r_0)>0$ such that
$$
e^{-C_7 s \log(1+ C_8 s)} \leq C_9 f(s), \quad s \geq r_0.
$$ 
\item[(d)] There is a constant $C_{10}=C_{10}(r_0)$ such that
\begin{align*}
\int_{|x-y|>r_0} f(|y-x|) \nu_r^{\ast n}(y)\, dy \leq \left(C_{10} \Psi\left(1/r \right)\right)^n f(|x|) ,\quad |x|\geq 3r_0, \ \ r \in (0,r_0], \ n \in \N.
\end{align*}
\item[(e)] There exists $C_{11}=C_{11}(r_0)$ such that for every $n \in \N$ and $r \in (0,r_0]$ we have
  \begin{align*}
    \nu_r^{n*}(x) \leq C_{11}^n \left[\Psi(1/r)\right]^{n-1} f(|x|), \quad |x|>3r_0.
  \end{align*}
\item[(f)] There exists $C_{12}=C_{12}(r_0)$ such that we have
  \begin{align*}
    \bar{p}_t(x) \leq C_{12} \, t \, f(|x|), \quad |x|>3r_0, \ t \in (0,t_0], 
  \end{align*}
	with $t_0:= 1/\Psi(1/r_0)$.
\end{itemize}
\end{lemma}

\begin{proof}
We first prove the first inequality in (a). Observe that by \eqref{ass:sjp} we have $K(1) < \infty$, 
which is equivalent to the existence of $c_1>0$ such that
$$
\int_{|x-y| > 1 \atop |y| > 1} f(|x-y|)f(|y|) dy \leq c_1 f(|x|), \quad |x| \geq 1.
$$
Since the profile $f$ is non-increasing and strictly positive, this implies that in 
fact for every $r_0>0$ there exists $c_2=c_2(r_0)$ satisfying
\begin{align} \label{eq:eq_aux}
\int_{|x-y| > r_0\atop |y| > r_0} f(|x-y|)f(|y|) dy \leq c_2 f(|x|), \quad |x| \geq 2r_0,
\end{align}
and the first inequality in (a) can be proved by following the lines of the proof of \cite[Lemma 3]{KSz2}.
The second inequality in (a) follows directly from \eqref{ass:low_reg} and \eqref{eq:Psi_eigenschaften}. 

To show (b), observe that by \eqref{eq:eq_aux} one has
$$
f(s) c_1 \geq f(s-r_0) \int_{|y-x_{r_0}| < r_0/2} f(y) dy =: c_3 f(s-r_0), \quad s \geq 3r_0, 
$$
with $ x_{r_0}:=((3r_0)/2,0,...,0)$.

The assertion (c) follows from the proof of part (b) of \cite[Lemma 1 (a)]{KSz2} and the assertions
(d) and (e) holds by \cite[Lemma 2]{KSz2} (the assumption $|\nu|=\infty$ is not needed now). Finally, 
(f) is a direct consequence of \cite[Lemma 4 (b)]{KSz2}. 
\end{proof}

We will need the following lemma. It gives a nontrivial result for $\kappa >0$. 

\begin{lemma}\label{lm:useful_2} 
Let the assumptions \textbf{(B)} and \textbf{(C)} hold with some $E \subset \sfera$ and $\kappa \geq 0$. Then for every fixed $r_0>0$ and $n \in \N$ one has
$$
\int_{\R^d} e^{\kappa (\scalp{\theta}{z})} \nu^{n*}_r(z) \, dz \leq (C_0 /C_1) \big(C_{10} \Psi(1/r) \big)^n, \quad r \in (0,r_0], \ \theta \in E,
$$ 
and
$$
\lim_{R \to \infty} \sup_{(r,\theta) \in (0,r_0] \times E} \int_{|z| > R} e^{\kappa (\scalp{\theta}{z})} \frac{\nu^{n*}_r(z)}{\Psi(1/r)^n} \, dz = 0, \quad n \in \N.
$$
In particular, if $\kappa > 0$, then \eqref{eq:v-moment} holds for every $\xi=k\theta$, with $\theta\in E$, the function $\theta \mapsto \widetilde \psi(\kappa \theta)$ is well defined and uniformly bounded on $E$, and \eqref{eq:hyp_exp_1} holds for $\xi=\kappa\theta$, $\theta \in E$. 
\end{lemma}

\begin{proof}
Fix $r_0 >0$ and $n \in \N$. By Lemma \ref{lm:useful} (d), Fatou's Lemma and the assumption \textbf{(C)}, we have
$$
  \int_{\R^d} e^{\kappa (\scalp{\theta}{z})} \nu^{n*}_r(z) \, dz 
	\leq \liminf_{s \to \infty} \int_{|s \theta-z|>r_0} \frac{\nu(s \theta-z)}{\nu(s \theta)} \nu_r^{n \ast}(z)\, dz 
	\leq \frac{C_0\left(C_{10} \Psi\left(1/r \right)\right)^n}{C_1}, \quad r \in (0,r_0], \ \theta \in E,
$$
which is exactly the first inequality. Moreover, by Lemma \ref{lm:useful} (e), for $r \in (0,r_0]$, $\theta \in E$, $s>1$ and $R > 3r_0$,
$$
  \int_{|s \theta-z|>R \atop |z| > R} \frac{\nu(s \theta-z)}{\nu(s \theta)} \frac{\nu_r^{n*}(z)}{\Psi\left(1/r \right)^n}\, dz 
	\leq \frac{C_{11}^n}{\Psi(1/r_0)}\int_{|s \theta-z|>R \atop |z| > R} \frac{\nu(s \theta-z)}{\nu(s \theta)} f(|z|)\, dz 
	\leq \frac{C_0 C_{11}^n}{C_1 \Psi(1/r_0)} K(r).
$$
Thus, by taking the $\liminf$ as $s \to \infty$ on both sides of the inequality and by applying Fatou's Lemma one more time, we get
$$
  0 \leq \int_{|z| > R} e^{\kappa (\scalp{\theta}{z})} \frac{\nu^{n*}_r(z)}{\Psi(1/r)^n} \, dz 
	\leq \frac{C_0 C_{11}^n}{C_1 \Psi(1/r_0)} K(r), \quad r \in (0,r_0], \ \ \theta \in E, \ \ R > 3r_0.
$$ 
Since the bound on the right hand side is uniform in $(r,\theta)$ on $(0,r_0] \times E$ and $K(r) \to 0$ as $R \to \infty$ by \textbf{(B)}, we get the claimed uniform convergence. 

The second assertion follows directly from the inequality proven above (with $n=1$), \cite[Th. 25.17]{Sato} and the Taylor expansion for the function $e^{\kappa (\scalp{\theta}{z})}$.
\end{proof}

We now discuss some known properties of the densities $\pzero_t$, which are used in the present paper. As proven in \cite[Lemma 8]{KSz1}, if \textbf{(D)} holds with some $T \subset (0,\infty)$, then there are constants $C_7$, $C_8$ and $C_9$ (dependent of $T$) such that
  \begin{equation}\label{eq:small_jumps_est}
    \pzero_t(x)\leq C_9 h(t)^{-d} 
    \exp\left[  \frac{-C_7|x|}{h(t)}\log\left(1+\frac{C_8|x|}{h(t)}\right)\right],\quad t \in T, \ x\in\Rd.
  \end{equation}
Denote
$$
F_T(r) := \exp\left[-C_7 r \log\left(1+C_8 r\right)\right], \quad r > 0.
$$
The subscript $T$ in the notation $F_T$ indicates that this function depends on a given set $T$ appearing in \textbf{(D)} via the constants $C_7$ and $C_8$. We will also need the following fact.

\begin{lemma} \label{lem:lambda_est}
Let the assumptions \textbf{(A)} and \textbf{(B)} hold. 
Then the following hold.
\begin{itemize}
\item[(a)] If $\inf_{|\xi|=1} \xi \cdot A \xi > 0$ in \textbf{(A)}, then for every $t_0 > 0$ there exists $R_0 > 0$ and $C_{13}, C_{14} >0$ such that
$$  
g_t(x) \leq C_{13} t f(|x|) e^{-C_{14}|x|^2}, \quad \text{as long as \ $t \in (0,t_0]$  \ and \ $|x| \geq R_0$.}
$$
\item[(b)] If, furthermore, \textbf{(D)} holds with some set $T \subset (0,\infty)$, then for every $r_0 > 0$ there exists $R_0 \geq r_0$ and $C_{15}, C_{16} >0$ such that
$$  
\lambda_t(x) \leq C_{15} t f(|x|) e^{-C_{16}|x| \log(1+C_{16}|x|)}, \quad \text{as long as \ $t \in T$, \ $h(t) \leq r_0$ \ and \ $|x| \geq R_0$,}
$$
where
$$
\lambda_t(x) = \left\{\begin{array}{ll}
\pzero_t(x) &\mbox{if } A \equiv 0, \\
\pzero_t \ast g_t (x) & \mbox{otherwise} \, .
\end{array}\right.
$$
\end{itemize}
\end{lemma}

\begin{proof}
We first prove (a). If $\inf_{|\xi|=1} \xi \cdot A \xi > 0$, then it is known that there exist $c_1, c_2 >0$ such that 
the following Aronson-type upper estimate holds
$$
  g(x) \leq c_1 t^{-d/2} e^{-c_2 \frac{|x|^2}{t}}, \quad t > 0, \ \ x \in \R^d.
$$
We may assume that $t_0 > 1$. Suppose first that $t \in [1, t_0]$. The bound above implies
$$
  g_t(x) \leq c_1 t e^{-(c_2/t_0) |x|^2}
$$
and we can easily find $R_0 > 0$ large enough such that for $|x| \geq R_0$ the function $e^{-\frac{c_2}{2t_0} |x|^2}$ is not bigger than $e^{-|x| \log(1+|x|)}$. Then, by Lemma \ref{lm:useful} (c), 
$$
  g_t(x) \leq c_3 t f(|x|) e^{-\frac{c_2}{2t_0} |x|^2}, \quad t \in [1, t_0], \  \ |x| \geq R_0. 
$$ 
On the other hand, if $0 < t < 1$, then 
$$
g_t(x) \leq c_1 t^{-d/2} e^{-c_2 \frac{|x|^2}{t}} \leq c_1 t \, \left(\frac{1}{\sqrt{t}}\right)^{d+2} e^{-\frac{c_2 r_0^2}{2} \left(\frac{1}{\sqrt{t}}\right)^2} e^{-c_2 \frac{|x|^2}{2}}, \quad |x| \geq R_0.
$$
Now, by increasing $R_0$ if necessary, we may get $e^{-c_2 \frac{|x|^2}{4}} \leq e^{-|x| \log(1+|x|)}$ for $|x| \geq R_0$, and again, by Lemma \ref{lm:useful} (c), 
$$
g_t(x) \leq c_4 t f(|x|) e^{-c_2 \frac{|x|^2}{4}}, \quad t \in (0,1), \ \ |x| \geq R_0. 
$$
This completes the proof of part (a). 

We now consider (b). Fix $r_0 >0 $ and assume first that $\inf_{|\xi|=1} \xi \cdot A \xi > 0$. We have
$$
\pzero_t \ast g_t (x) = \int_{|z| > \frac{|x|}{4}} \pzero_t(x-z) g_t(z) dz + \int_{|z| \leq \frac{|x|}{4}} \pzero_t(x-z) g_t(z) dz \leq \sup_{|z| > \frac{|x|}{4}} g_t(z) + \sup_{|z| > \frac{3}{4}|x|} \pzero_t(z)
$$
and it is enough to estimate both suprema on the right hand side for large $|x|$. It follows from the part (a) that we can find $R_0 \geq r_0$ large enough such that for $|x| \geq R_0$ we have $\sup_{|z| > \frac{|x|}{4}} g_t(z) \leq c_5 t f(|x|) e^{-|x| \log(1+|x|)}$, whenever $h(t) \leq r_0$. To deal with the second supremum, we note that by \eqref{eq:small_jumps_est} we have 
$$
\sup_{|z| > \frac{3}{4}|x|} \pzero_t(z) \leq C_9 \left(\frac{1}{h(t)}\right)^{d} e^{-C_7 \frac{R_0}{4h(t)}} F_T\left(\frac{|x|}{4r_0}\right) F_T\left(\frac{|x|}{4r_0}\right), \quad t \in T, \ |x| \geq R_0.
$$
It follows from \cite[Lemma 3.6.22]{J1} that $\Psi(r) \leq 2\Psi(1)(1+r^2)$, $ r > 0$, which implies that
$$
e^{-C_7 \frac{R_0}{4h(t)}} \leq c_7 h(t)^{d+2} = \frac{c_8 h(t)^{d+2} }{1+r_0^2} 
\leq \frac{c_8 h(t)^{d+2} }{1+h(t)^2} = \frac{c_8 h(t)^d }{1+\frac{1}{h(t)^2}} \leq \frac{c_9 h(t)^d }{\Psi\left(\frac{1}{h(t)}\right)} = c_9 t h(t)^d
$$
as long as $h(t) \leq r_0$. This and Lemma \ref{lm:useful} (c) finally give
$$
\sup_{|z| > \frac{3}{4}|x|} \pzero_t(z) \leq c_{10} t f(|x|) e^{-c_{11}|x| \log(1+c_{11}|x|)}, \ t \in T, \ h(t) \leq r_0, \ |x| \geq R_0,
$$
which completes the proof of (b) in the case $\inf_{|\xi|=1} \xi \cdot A \xi > 0$. The proof of (b) in case $A \equiv 0$ follows directly that the argument leading to the upper bound of $\sup_{|z| > \frac{3}{4}|x|} \pzero_t(z)$ above and is omitted.
\end{proof}

\bigskip

\section{Asymptotics of convolutions of the L\'evy measures }

The following two lemmas will be basic for our further investigations.

\begin{lemma} \label{lem:conv}
Let the assumptions \textbf{(B)} and \textbf{(C)} hold with some $E \subset \sfera$ and $\kappa \geq 0$. Moreover, let $r_0>0$ be arbitrary. Then the following hold. 
\begin{itemize}
\item[(a)]
For every $n \in \N$, $r \in (0,r_0]$, $\theta \in E$ and $y \in \R^d$
\begin{equation}\label{eq:conv_limit_0_0}
	\lim_{s \to \infty} \frac{\nu_r^{n*}(s \theta-y)}{\nu_r(s \theta)} = 
	e^{\kappa (\scalp{\theta}{y})} n \left(\int_{|z|>r}e^{\kappa (\scalp{\theta}{z})}\nu(z)\, dz\right)^{n-1}.
\end{equation}
\item[(b)] If the convergence in \textbf{(C)} is uniform in $(\theta,y)$ on each rectangle $E \times D$, for every compact set $D \subset \R^d$, then for any $n \in \N$ the convergence \begin{equation}\label{eq:conv_limit_0}
	\lim_{s \to \infty} \frac{\nu_r^{n*}(s \theta-y)}{\nu_r(s \theta)\Psi(1/r)^{n-1}} = 
	e^{\kappa (\scalp{\theta}{y})} n \left(\frac{\int_{|z|>r}e^{\kappa (\scalp{\theta}{z})}\nu(z)\, dz}{\Psi(1/r)}\right)^{n-1}
\end{equation}  
is uniform in $(r,\theta,y)$ on each cuboid $(0,r_0] \times E \times B(0,\varrho)$, $\varrho >0$.
\end{itemize}
\end{lemma}

\begin{proof}
Let $E \subset \sfera$ and $\kappa \geq 0$ be as in the assumption \textbf{(C)}. Fix $r_0>0$. 

\vspace{0.1cm}

\noindent
(a) We first establish the pointwise convergence. The argument is based on induction on $n$. For $n=1$ the assertion is just the assumption \textbf{(C)}. Suppose now that the convergence in \eqref{eq:conv_limit_0_0} (or, equivalently, the pointwise convergence in \eqref{eq:conv_limit_0}) holds for some $n \in \N$ and every $\theta \in E$ and $y \in \R^d$. Fix $y \in \R^d$, $\theta \in E$ and $r \in (0,r_0]$. Denote
$$
V_{n,r,\theta,y}(s) := \left|\frac{\nu_r^{(n+1)*}(s\theta-y)}{[\Psi(1/r)]^{n} \nu(s \theta)} - \, e^{\kappa (\scalp{\theta}{y})} (n+1) \left(\frac{\int_{\R^d}e^{\kappa (\scalp{\theta}{z})}\nu_r(z)\, dz}{\Psi(1/r)}\right)^{n}\right|, \quad s >0.
$$
We will prove that $\lim_{s \to \infty} V_{n,r,\theta,y}(s) = 0$. Let $R>3r_0$ and $s>2R+|y|$. Observe that
\begin{align*}
\nu_r^{(n+1)*}(s\theta-y) & = \int_{|z| \leq R} \nu_r(s\theta-y-z) \nu^{n*}_r(z) \, dz \\ & + \int_{|s\theta-y-z| > R \atop |z| > R} \nu_r(s\theta-y-z)\nu^{n*}_r(z) dz + \int_{|w| \leq R} \nu^{n*}_r(s\theta - y - w) \nu_r(w) \, dw
\end{align*}
and 
$$
	e^{\kappa (\scalp{\theta}{y})} (n+1) \left(\frac{\int_{\R^d}e^{\kappa (\scalp{\theta}{z})}\nu_r(z)\, dz}{\Psi(1/r)}\right)^{n} 
	= \frac{\int_{\R^d}e^{\kappa \big(\scalp{\theta}{(y+z)}\big)}\nu^{n*}_r(z)\, dz}{[\Psi(1/r)]^n}
	+  e^{\kappa (\scalp{\theta}{y})} n \left(\frac{\int_{\R^d}e^{\kappa (\scalp{\theta}{z})}\nu_r(z)\, dz}{\Psi(1/r)}\right)^{n},
$$
since
		\begin{eqnarray*}
		\int_{\R^d} e^{\kappa (\scalp{\theta}{z})} \nu_r^{n*}(z)\, dz
		&  =   & \int_{\R^d}\int_{\R^d}...\int_{\R^d} e^{\kappa \big(\scalp{\theta}{(z_1+...+z_n)}\big)} \, 
		         \nu_r(dz_1)\nu_r(dz_2)...\nu_r(dz_n) \\
		&  =   & \left( \int e^{\kappa (\scalp{\theta}{z})} \nu_r(z)\, dz\right)^n.
	\end{eqnarray*}
Thus, for arbitrary $R>3r_0$ and $s > 2R + |y|$, we have
	\begin{align*}
 & V_{n,r,\theta,y}(s) \\ & \leq \underbrace{\int_{|z| \leq R} \left|\frac{\nu_r(s\theta-y-z)}{\nu(s\theta)} - e^{\kappa \big(\scalp{\theta}{(y+z)}\big)}\right|  \frac{\nu^{n*}_r(z)}{[\Psi(1/r)]^n} \, dz}_{=: I_1(s,r,\theta,y,R)} + \, e^{\kappa (\scalp{\theta}{y})} \underbrace{\int_{|z| > R} e^{\kappa (\scalp{\theta}{z})} \frac{\nu^{n*}_r(z)}{[\Psi(1/r)]^n} \, dz}_{=: I_2(r,\theta,R)}\\
		& + \underbrace{\int_{|w| \leq R} \left|\frac{\nu^{n*}_r(s\theta - y - w)}{\nu(s\theta) \, [\Psi(1/r)]^{n-1}} - e^{\kappa \big(\scalp{\theta}{(y+w)}\big)} n \left(\frac{\int_{\R^d}e^{\kappa (\scalp{\theta}{z})}\nu_r(z)\, dz}{\Psi(1/r)}\right)^{n-1}\right| \, \frac{\nu_r(w)}{\Psi(1/r)} \, dw }_{=: I_3(s,r,\theta,y,R)}  \\
		& + e^{\kappa (\scalp{\theta}{y})} n \left(\frac{\int_{\R^d}e^{\kappa (\scalp{\theta}{z})}\nu_r(z)\, dz}{\Psi(1/r)}\right)^{n-1}
		        \underbrace{ \int_{|w| >R} e^{\kappa (\scalp{\theta}{w})} \frac{\nu_r(w)}{\Psi(1/r)} dw }_{=: I_4(r,\theta,R)}+  \underbrace{\int_{|s\theta-y-z| > R \atop |z| > R} \frac{\nu_r(s\theta-y-z)\nu^{n*}_r(z)}{\nu(s\theta) [\Psi(1/r)]^n} dz }_{=: I_5(s,r,\theta,y,R)},
	\end{align*}	
which can be rewritten in short as
	\begin{align} \label{eq:V_est}
V_{n,r,\theta,y}(s) & \leq I_1(s,r,\theta,y,R) + e^{\kappa |y|} I_2(r,\theta,R) + I_3(s,r,\theta,y,R) \\ & \ \ \ \ + e^{\kappa|y|} n \left(\frac{\int_{\R^d}e^{\kappa (\scalp{\theta}{z})}\nu_r(z)\, dz}{\Psi(1/r)}\right)^{n-1} I_4(r,\theta,R)+ I_5(s,r,\theta,y,R). \nonumber
	\end{align}	
We first consider $I_1$ and $I_3$. Observe that by \eqref{eq:add_conv} and the induction hypothesis \eqref{eq:conv_limit_0_0}, both integrands under these two integrals go to zero pointwise as $s \to \infty$, for any $z$ and $w \in B(0,R)$, respectively. Moreover, by Lemma \ref{lm:useful} (b) and (e) and Lemma \ref{lm:useful_2}, there exist constants $c_1$, $c_2=c_2(|y|,R)$, $c_3=c_3(r_0,n)$, $c_4=c_4(r_0,n)$ and $c_5=c_5(|y|,r_0,R, n)$ such that for $s\geq 3(R+|y|)$ we have
$$
\left|\frac{\nu_r(s\theta-y-z)}{\nu(s\theta)} - e^{\kappa \big(\scalp{\theta}{(y+z)}\big)}\right| 
\leq c_1 \frac{f(|s\theta - y -z|)}{f(s)} + e^{\kappa (|y|+|z|)} \leq c_2 + e^{\kappa (|y|+R)}, \quad |z| \leq R,
$$
and 
\begin{align*}
\left|\frac{\nu^{n*}_r(s\theta - y - w)}{\nu(s\theta) \, [\Psi(1/r)]^{n-1}} - e^{\kappa \big(\scalp{\theta}{(y+w)}\big)} n \left(\frac{\int_{\R^d}e^{\kappa (\scalp{\theta}{z})}\nu_r(z)\, dz}{\Psi(1/r)}\right)^{n-1}\right| & \leq c_3 \frac{f(|s\theta - y -z|)}{f(s)} + c_4 e^{\kappa (|y|+|w|)} \\ & \leq c_5 + c_4 e^{\kappa (|y|+R)}, \quad |w| \leq R.
\end{align*}
Therefore, by bounded convergence, both integrals $I_1(s,r,\theta,y,R)$ and $I_3(s,r,\theta,y,R)$ tend to $0$ as $s \to \infty$, for $r \in (0,r_0]$, $\theta \in E$ and $y \in \R^d$.

To deal with $I_5$, it is enough to observe that by Lemma \ref{lm:useful} (e) and (b) and by the definition of the function $K$ in \textbf{(B)} one has
$$
I_5(s,r,\theta,y,R) \leq \frac{c_6}{\Psi(1/r_0)} \int_{|s\theta-y-z| > R \atop |z| > R} \frac{f(|s\theta-y-z|) f(|z|)}{f(|s\theta-y|)} \frac{f(|s\theta-y|)}{f(|s\theta|)}dz \leq c_7 K(r),
$$
with some $c_6=c_6(r_0,n)$ and $c_7=c_7(r_0,n)$.

Therefore, \eqref{eq:V_est}, all the above observations taken together, Lemma \ref{lm:useful_2} and our basic assumption \textbf{(B)} give
\begin{align*}
\limsup_{s \to \infty} V_{n,r,\theta,y}(s) & \leq e^{\kappa |y|} I_2(r,\theta,R) + e^{\kappa|y|} n (C_0C_{10}/C_1)^{n-1} I_4(r,\theta,R)+ c_7 K(r)
\end{align*}
and, letting $R \to \infty$, finally $\lim_{s \to \infty} V_{n,r,\theta,y}(s) = 0$. Since $r \in (0,r_0]$, $\theta \in E$ and $y \in \R^d$ were choosen arbitrarily, this completes the proof of the part (a). 

\medskip

\noindent
(b) Assume now, moreover, that the convergence in \textbf{(C)} is uniform in $(\theta,y)$ on each rectangle $E \times D$, for every compact set $D \subset \R^d$. We again use a induction on $n$. Observe that similarly as before, for $n=1$ the assertion follows directly from the assumption \textbf{(C)}. Suppose that for some $n \in \N$ the convergence in \eqref{eq:conv_limit_0} holds uniformly in $(r,\theta,y)$ on each cuboid $(0,r_0] \times E \times B(0,\varrho)$, $\varrho >0$. We have to prove that for every $\varrho > 0$ one has $\lim_{s \to \infty} \sup_{(r,\theta,y) \in (0,r_0] \times E \times B(0,\varrho)}V_{n,r,\theta,y}(s) = 0$. Observe that by following the estimates in part (a), we only need to show that
\begin{align} \label{eq:i1andi2}
\lim_{s \to \infty} \sup_{(r,\theta,y) \in (0,r_0] \times E \times B(0,\varrho)} \big( I_1(s,r,\theta,y,R) + I_3(s,r,\theta,y,R)\big) = 0, 
\end{align}
for all $\varrho >0$ and $R>0$. Indeed, if this is true, then similarly as in (a) we get from \eqref{eq:V_est} that
\begin{align*}
\limsup_{s \to \infty} & \left[\sup_{(r,\theta,y) \in (0,r_0] \times E \times B(0,\varrho)} V_{n,r,\theta,y}(s) \right] \\ & \leq e^{\kappa \varrho} \sup_{(r,\theta) \in (0,r_0] \times E} I_2(r,\theta,R) + e^{\kappa \varrho } n (C_0C_{10}/C_1)^{n-1} \sup_{(r,\theta) \in (0,r_0] \times E} I_4(r,\theta,R)+  c_7K(r).
\end{align*}
An application of Lemma \ref{lm:useful_2} and the assumption \textbf{(B)} gives that the members on the right hand side go to zero as $R \to \infty$, which is exactly our claim.

To this end, we will show \eqref{eq:i1andi2}. Fix $\varrho >0$ and let $y \in B(0,\varrho)$.
We have
\begin{align*}
\sup_{(r,\theta,y) \in (0,r_0] \times E \times B(0,\varrho)}  I_1(s,r,\theta,y,R) & \leq \frac{\nu^{n*}_r(B(0,R))}{[\Psi(1/r)]^n} \sup_{(r,\theta,w) \in (0,r_0] \times E \times B(0,\varrho+R)}\left|\frac{\nu_r(s\theta-w)}{\nu(s\theta)} - e^{\kappa (\scalp{\theta}{w})}\right| \\
& \leq C_3^{n} \sup_{(r,\theta,w) \in (0,r_0] \times E \times B(0,\varrho+R)}\left|\frac{\nu_r(s\theta-w)}{\nu(s\theta)} - e^{\kappa (\scalp{\theta}{w})}\right|
\end{align*}
and, similarly,
\begin{align*}
\sup_{(r,\theta,y) \in (0,r_0] \times E \times B(0,\varrho)} & I_3(s,r,\theta,y,R) \\ &  \leq C_3\sup_{(r,\theta,w) \in (0,r_0] \times E \times B(0,\varrho+R)}\left|\frac{\nu^{n*}_r(s\theta - w)}{\nu(s\theta) \, [\Psi(1/r)]^{n-1}} - e^{\kappa (\scalp{\theta}{w})} n \left(\frac{\int_{\R^d}e^{\kappa (\scalp{\theta}{z})}\nu_r(z)\, dz}{\Psi(1/r)}\right)^{n-1}\right|.
\end{align*}
By induction hypothesis both suprema on the right hand side of the above inequalities tend to zero as $s \to \infty$, which completes the proof of \eqref{eq:i1andi2} and the proof of the entire lemma.
\end{proof}

Recall that 
$$
\bar{p}_t(x) := \left[e^{-t |\nu_{r}|} \sum_{n=1}^{\infty} \frac{\nu_{r}^{n \ast}(x)}{[\Psi(1/r)]^n n!}\right]_{r = h(t)}  = e^{-t |\nu_{h(t)}|} \sum_{n=1}^{\infty} \frac{t^n \nu_{h(t)}^{n \ast}(x)}{n!}, \quad x \in \R^d, \ t > 0.
$$
The following result on the spatial asymptotics of the density $\bar{p}_t(x)$ is a consequence of Lemma \ref{lem:conv}.

\begin{lemma} \label{lem:conv_Poiss}
Let the assumptions \textbf{(B)}-\textbf{(C)} hold with some $E \subset \sfera$ and $\kappa \geq 0$. Moreover, let $r_0>0$ be arbitrary and denote $t_0 := 1/\Psi(1/r_0)$.
\begin{itemize}
\item[(a)]
For every $t \in (0,t_0]$, $\theta \in E$ and $y \in \R^d$ one has
\begin{equation}\label{eq:conv_Poisson}
	\lim_{s \to \infty} \frac{\bar{p}_t(s \theta-y)}{t \, \nu(s \theta)} = 
	\exp\left(\kappa (\scalp{\theta}{y}) \right) \, \exp\left( t \int_{|z|>h(t)}\left(e^{\kappa (\scalp{\theta}{z})} - 1\right) \, \nu(z)\, dz\right).
\end{equation} 
\item[(b)] If the convergence in \textbf{(C)} is uniform in $(\theta,y)$ on each rectangle $E \times D$, for every compact set $D \subset \R^d$, then the convergence in \eqref{eq:conv_Poisson} is uniform in $(t,\theta,y)$ on each cuboid $(0,t_0] \times E \times B(0,\varrho)$, $\varrho >0$.
\end{itemize}
\end{lemma}

\begin{proof} 
(a) For $t \in (0,t_0]$, $\theta \in E$ and $y \in \R^d$ one may write
\begin{align*}
\left|\frac{\bar{p}_t(s \theta-y)}{t \, \nu(s \theta)} \right. & \left. - \exp\left(\kappa (\scalp{\theta}{y}) \right) \, \exp\left( t \int_{|z|>h(t)}\left(e^{\kappa (\scalp{\theta}{z})} - 1\right) \, \nu(z)\, dz\right)\right| \\
& \leq e^{-t |\nu_{h(t)}|} \sum_{n =1}^{\infty} \frac{1}{n!} \underbrace{\left|\frac{\nu^{n*}_{h(t)}(s\theta - y)}{\nu(s\theta) \, [\Psi(1/h(t))]^{n-1}} - e^{\kappa (\scalp{\theta}{y})} n \left(\frac{\int_{\R^d}e^{\kappa (\scalp{\theta}{z})}\nu_{h(t)}(z)\, dz}{\Psi(1/h(t))}\right)^{n-1}\right|}_{=:J(s,t,\theta,y,n)}.
\end{align*}
By Lemma \ref{lem:conv} (a), $\lim_{s \to \infty} J(s,t,\theta,y,n) = 0$, for every $t \in (0,t_0]$, $\theta \in E$, $y \in \R^d$ and $n \in \N$. Moreover, by Lemmas \ref{lm:useful} (e) and \ref{lm:useful_2}, for every $t \in (0,t_0]$, $\theta \in E$, $y \in \R^d$ and $n \in \N$, 
$$
J(s,t,\theta,y,n) \leq (C_0/C_1) C_{11}^n + e^{\kappa |y|} n (C_0 C_{10}/C_1)^{n-1}, 
$$
and, therefore, by dominated convergence, the above series tends to zero as $s \to \infty$ giving (a).

To show (b), fix $\varrho >0$ and observe that
\begin{align*}
\sup_{(t,\theta,y) \in (0,t_0] \times E \times B(0,\varrho)} \left|\frac{\bar{p}_t(s \theta-y)}{t \, \nu(s \theta)} \right. & \left. - \exp\left(\kappa (\scalp{\theta}{y}) \right) \, \exp\left( t \int_{|z|>h(t)}\left(e^{\kappa (\scalp{\theta}{z})} - 1\right) \, \nu(z)\, dz\right)\right| \\
& \leq \sum_{n =1}^{\infty} \frac{1}{n!} \sup_{(t,\theta,y) \in (0,t_0] \times E \times B(0,\varrho)} J(s,t,\theta,y,n).
\end{align*}
Since, for every $n \in \N$, 
$$
\sup_{(t,\theta,y) \in (0,t_0] \times E \times B(0,\varrho)} J(s,t,\theta,y,n) \leq (C_0/C_1) C_{11}^n + e^{\kappa \varrho} n (C_0 C_{10}/C_1)^{n-1}
$$
and 
\begin{align*}
& \sup_{(t,\theta,y) \in (0,t_0] \times E \times B(0,\varrho)} J(s,t,\theta,y,n)  \\ & \ \ \ \ \  \leq \sup_{(r,\theta,y) \in (0,r_0] \times E \times B(0,\varrho)} \left|\frac{\nu^{n*}_{r}(s\theta - y)}{\nu(s\theta) \, [\Psi(1/r)]^{n-1}} - e^{\kappa (\scalp{\theta}{y})} n \left(\frac{\int_{\R^d}e^{\kappa (\scalp{\theta}{z})}\nu_{r}(z)\, dz}{\Psi(1/r)}\right)^{n-1}\right| \to 0
\end{align*}
as $s \to \infty$ (by Lemma \ref{lem:conv} (b)), also the claim of the part (b) holds true. 
\end{proof}

\bigskip

\section{Proof of the main result}

\begin{proof}[Proof of Theorem \ref{thm:main}]
Let the assumptions \textbf{(A)} and \textbf{(B)} hold and let $b \in \R^d$ be arbitrary. 
Let $r_0 >0$ be so large that $t_0:= 1/\Psi(1/r_0) \geq \sup T$, where $T$ is the subset of $\R$ given in assumption \textbf{(D)}, and let $E \subset \sfera$ and $\kappa \geq 0$ be the subset of a unit sphere and the number appearing in assumption \textbf{(C)}. Moreover, let 
$R_0 \geq r_0$ be the radius provided by Lemma \ref{lem:lambda_est} (b).

\vspace{0.1cm}

\noindent
(a) We will show that for every $t \in T$, $\theta \in E$ and $y \in \R^d$,
\begin{align} \label{eq:res_conv_initial}
\lim_{r \to \infty} \frac{p_t(r\theta -y + tb_{h(t)})}{t \, \nu(r\theta)} =  e^{\kappa (\scalp{\theta}{y})} \exp\left(t\int_{|z|>h(t)} \left( e^{\kappa (\scalp{\theta}{z})} - 1 \right)\nu(z) dz\right) \, \int_{\R^d} e^{\kappa (\scalp{\theta}{z})} \lambda_t(dz),
\end{align}
where 
$$
\lambda_t(dz) = \left\{\begin{array}{ll}
\widetilde p_t(z)dz &\mbox{if } A \equiv 0, \\
\widetilde p_t \ast g_t (z)dz & \mbox{otherwise} \, .
\end{array}\right.
$$
If this is true, then by substituting $y = w + t b_{h(t)}$ (recall that $b_{r}$ is defined in \eqref{eq:def_br}) the right hand side of 
\eqref{eq:res_conv_initial} becomes $e^{-t \widetilde \psi(\kappa \theta) + \kappa (\scalp{\theta}{w})}$ with $\widetilde \psi$ given by \eqref{eq:hyp_exp} and we get exactly the assertion (a) of the theorem. Indeed, $\int_{\R^d} e^{\kappa (\scalp{\theta}{z})} \lambda_t(dz)$ is a multidimensional exponential moment of order $\kappa \theta$ of the measure $\lambda_t(dz)$ and, according to \eqref{eq:hyp_exp}-\eqref{eq:hyp_exp_1}, we have
$$
\int_{\R^d} e^{\kappa (\scalp{\theta}{z})} \lambda_t(dz) =  \exp\left(t\int_{\R^d \setminus \left\{0\right\}} 
		\left(e^{\kappa (\scalp{\theta}{y})}-1-\kappa(\scalp{\theta}{y}) \right) \nuzero_{h(t)}(y)\, dy\right), \quad t>0,
$$
when $A \equiv 0$, and 
$$
\int_{\R^d} e^{\kappa (\scalp{\theta}{z})} \lambda_t(dz) = \exp\left(t(\scalp{\xi}{A \xi}) + t \int_{\R^d \setminus \left\{0\right\}} 
		\left(e^{\kappa (\scalp{\theta}{y})}-1-\kappa(\scalp{\theta}{y}) \right) \nuzero_{h(t)}(y)\, dy\right), \quad t>0,
$$
otherwise.

Denote for shorthand
$$
W(s,t,\theta,y) := \left|\frac{p_t(s \theta - y + tb_{h(t)})}{t \, \nu(s \theta)} - e^{\kappa (\scalp{\theta}{y})} \exp\left(t\int_{|z|>h(t)} \left( e^{\kappa (\scalp{\theta}{z})} - 1 \right)\nu(z) dz\right) \, \int_{\R^d} e^{\kappa (\scalp{\theta}{z})} \lambda_t(dz)\right|.
$$
With this notation our goal is to show that $\lim_{s \to \infty} W_{T}(s,t,\theta,y) = 0$, for $t \in T$, $\theta \in E$ and $y \in \R^d$. 

Fix $t \in T$, $\theta \in E$, $y \in \R^d$. Recall that by \eqref{eq:basic_conv_1}-\eqref{eq:basic_conv_2}, for every $R>3r_0 \vee R_0$ and $s> 2R + |y|$, we may write
\begin{align*}
\frac{p_t(s \theta - y + tb_{h(t)})}{t \, \nu(s \theta)} & = \frac{e^{-t|\nu_{h(t)}|} \lambda_t(s \theta - y)}{t \, \nu(s\theta)} + \frac{\lambda_t \ast \bar{p}_t (s \theta - y)}{t \, \nu(s \theta)} \\ 
& = \frac{e^{-t|\nu_{h(t)}|} \lambda_t(s \theta - y)}{t \, \nu(s\theta)} \\ & \ \ \ \ +
\left(\int_{|z| \leq R} + \int_{|s\theta -y - z| \leq R} + \int_{|s \theta - y - z| > R \atop |z| > R}\right) \frac{\bar{p}_t (s \theta - y - z)}{t \, \nu(s \theta)} \lambda_t(dz).
\end{align*}
With this, we have
\begin{align*}
& \frac{p_t(s \theta - y + tb_{h(t)})}{t \, \nu(s \theta)} - e^{\kappa (\scalp{\theta}{y})} \exp\left(t\int_{|z|>h(t)} \left( e^{\kappa (\scalp{\theta}{z})} - 1 \right)\nu(z) dz\right) \, \int_{\R^d} e^{\kappa (\scalp{\theta}{z})} \lambda_t(dz)  \\ 
& = \frac{e^{-t|\nu_{h(t)}|} \lambda_t(s \theta-y)}{t \, \nu(s\theta)} + \frac{\lambda_t \ast \bar{p}_t (s \theta-y)}{t \, \nu(s \theta)} \\ 
& = \frac{e^{-t|\nu_{h(t)}|} \lambda_t(s \theta-y)}{t \, \nu(s\theta)} +
\int_{|s\theta - y-z| \leq R} \frac{\bar{p}_t (s \theta - y-z)}{t \, \nu(s \theta)} \lambda_t(z) dz + \int_{|s \theta - y-z| > R \atop |z| > R} \frac{\bar{p}_t (s \theta - y-z)}{t \, \nu(s \theta)} \lambda_t(z) dz \\
& \ \ \ + \int_{|z| \leq R} \left(\frac{\bar{p}_t (s \theta - y-z)}{t \, \nu(s \theta)} -  e^{\kappa (\scalp{\theta}{(y+z)})} \, \exp\left( t \int_{|w|>h(t)}\left(e^{\kappa (\scalp{\theta}{w})} - 1\right) \, \nu(w)\, dw\right)\right) \lambda_t(z) dz \\
& \ \ \ + e^{\kappa (\scalp{\theta}{y})}\int_{|z|>R} e^{\kappa (\scalp{\theta}{z})} \, \exp\left(t\int_{|w|>h(t)} \left( e^{\kappa (\scalp{\theta}{w})} - 1 \right)\nu(w) dw\right) \, \lambda_t(z) dz,
\end{align*}
which leads to the inequality
\begin{align*}
& \left|\frac{p_t(s \theta - y + tb_{h(t)})}{t \, \nu(s \theta)} - e^{\kappa (\scalp{\theta}{y})} \exp\left(t\int_{|z|>h(t)} \left( e^{\kappa (\scalp{\theta}{z})} - 1 \right)\nu(z) dz\right) \, \int_{\R^d} e^{\kappa (\scalp{\theta}{z})} \lambda_t(dz)\right| \\ 
& \leq \underbrace{\frac{e^{-t|\nu_{h(t)}|} \lambda_t(s \theta-y)}{t \, \nu(s\theta)}}_{=:I_1(s,t,\theta,y)} +
\underbrace{\int_{|s\theta - y-z| \leq R} \frac{\bar{p}_t (s \theta - y-z)}{t \, \nu(s \theta)} \lambda_t(z) dz}_{=:I_2(s,t,\theta,y,R)} + \underbrace{\int_{|s \theta - y-z| > R \atop |z| > R} \frac{\bar{p}_t (s \theta - y-z)}{t \, \nu(s \theta)} \lambda_t(z) dz}_{=:I_3(s,t,\theta,y,R)} \\
& \ \ \ + \underbrace{\int_{|z| \leq R} \left|\frac{\bar{p}_t (s \theta - y-z)}{t \, \nu(s \theta)} -  e^{\kappa (\scalp{\theta}{(y+z)})} \, \exp\left( t \int_{|w|>h(t)}\left(e^{\kappa (\scalp{\theta}{w})} - 1\right) \, \nu(w)\, dw\right)\right| \lambda_t(z) dz}_{=:I_4(s,t,\theta,y,R)} \\
& \ \ \ + e^{\kappa (\scalp{\theta}{y})} \, \underbrace{\exp\left(t\int_{|w|>h(t)} \left( e^{\kappa (\scalp{\theta}{w})} - 1 \right)\nu(w) dw\right) \, \int_{|z|>R} e^{\kappa (\scalp{\theta}{z})} \, \lambda_t(z) dz}_{=:I_5(t,\theta,R)},
\end{align*}
that is,
\begin{align} \label{eq:W_est}
W(s,t,\theta,y) \leq I_1(s,t,\theta,y) + I_2(s,t,\theta,y,R) + I_3(s,t,\theta,y,R) + I_4(s,t,\theta,y,R) + I_5(t,\theta,R) e^{\kappa |y|}
\end{align}
in short.

Let now $\varrho >0$ be an arbitrary positive number such that $|y| \leq \varrho$. We first estimate $I_1$ and $I_2$. By Lemma \ref{lem:lambda_est} (b) and Lemma \ref{lm:useful} (b) we have
$$
I_1(s,t,\theta,y) \leq c_1 \frac{f(s-|y|)}{f(s)} e^{-c_2 (s-|y|) \log(1+ c_2 (s-|y|))} \leq c_3  e^{-c_2 (s-\varrho) \log(1+ c_2 (s-\varrho))},
$$
and
\begin{align*}
I_2(s,t,\theta,y,R) = \int_{|z| \leq R} \bar{p}_t (z) \frac{\lambda_t(s\theta - y-z)}{t \, \nu(s \theta)}  dz
                    & \leq c_1 \frac{f(s-|y|-R)}{f(s)} e^{-c_2 (s-|y|-R) \log(1+ c_2 (s-|y|-R))} \\
										& \leq c_3 e^{-c_2 (s-\varrho-R) \log(1+ c_2 (s-\varrho-R))},
\end{align*}
with some $c_1=c_1(T,R)$, $c_2=c_2(r_0,T)$ and $c_3=c_3(\varrho,R)$, uniformly in  $s > 3(R + \varrho)$. 

To deal with $I_3$, we observe that by Lemmas \ref{lem:lambda_est} (b) and Lemma \ref{lm:useful} (b), (f), we get
\begin{align*}
I_3(s,t,\theta,y, R) & = \int_{|s \theta - y-z| > R \atop |z| > R} \frac{\bar{p}_t (s \theta - y-z)}{t \, \nu(s \theta)} \lambda_t(z) dz \\
              & \leq c_4 t_0 \frac{f(s-|y|)}{f(s)} \int_{|s \theta - y-z| > R \atop |z| > R} \frac{f(|s \theta - y-z|) f(|z|)}{f(|s \theta - y|)} dz \leq c_5 K(r),
\end{align*}
with $c_4 = c_4(r_0,T)$ and $c_5 = c_5(r_0,T,\varrho)$, uniformly in $s> 3(R + \varrho)$. 

Also, by Lemma \ref{lem:conv_Poiss} (a) the integrand under $I_4$ tends to zero as $s \to 0$. Moreover, since $t = 1/\Psi(1/h(t))$, Lemma \ref{lm:useful_2} yields
\begin{align} \label{eq:from_Lem2}
\exp\left(t\int_{|w|>h(t)} \left( e^{\kappa (\scalp{\theta}{w})} - 1 \right)\nu(w) dw \right)\leq e^{C_0C_{10}/C_1}. 
\end{align}
This together with Lemma \ref{lm:useful} (b), (f) implies that
$$
\left|\frac{\bar{p}_t (s \theta - y-z)}{t \, \nu(s \theta)} -  e^{\kappa (\scalp{\theta}{(y+z)})} \, \exp\left( t \int_{|w|>h(t)}\left(e^{\kappa (\scalp{\theta}{w})} - 1\right) \, \nu(w)\, dw\right)\right| 
\leq c_6 + e^{C_0 C_{10}/C_1+ \kappa(R+|y|)},
$$
for $s \geq 3r_0 + R + |y|$ , with $c_6=c_6(r_0,T,\varrho)$. Thus $\lim_{s \to \infty} I_4(s,t,\theta,y,R) = 0$, by bounded convergence. 

One more use of Lemma \ref{lem:lambda_est} (b) and \eqref{eq:from_Lem2} also gives that there exists $c_7=c_7(r_0,T)$ and $c_8=c_8(r_0,T)$ such that
$$
I_5(t,\theta, R) \leq c_7 t_0 e^{C_{10}} \int_{|z|>R} e^{-c_8 \big(\log(1+c_8 R) - \frac{\kappa}{c_8}\big)|z|} f(|z|) dz
$$

Collecting all the above observations, we get
\begin{align} \label{eq:W_est_2}
\limsup_{s \to \infty} W(s,t,\theta,y) \leq c_5 K(r) + I_5(t,\theta,R) e^{\kappa |y|},
\end{align}
and by taking the limit $R \to \infty$ we obtain that $\lim_{s \to \infty} W(s,t,\theta,y) = 0$, for every $t \in T$, $\theta \in E$ and $y \in \R^d$. This completes the proof of the part (a). 

\vspace{0.1cm}

\noindent
(b) Suppose now that the convergence in \textbf{(C)} is uniform in $(\theta,y)$ on each rectangle $E \times D$, for every compact set $D \subset \R^d$. Since, by \eqref{eq:def_br} and \eqref{eq:Psi_eigenschaften}, for every $t \in T \subset (0,t_0]$, one has $|t b_{h(t)}| \leq c_8$ with $c_8=c_8(r_0,|b|,|\nu_1|)$, as before, it is sufficient to prove that \eqref{eq:res_conv_initial} holds uniformly in $(t,\theta,y)$ on each cuboid $T \times E \times B(0,\varrho) $, $\varrho >0$.

Fix $\varrho >0$ and observe that by the estimates established in part (a) we have
\begin{align*} %\label{eq:W_est_3}
\sup_{(t,\theta,y) \in \ T \times E\times B(0,\varrho)} W(s,t,\theta,y) & \leq c_1 e^{-c_2 (s-\varrho) \log(1+ c_2 (s-\varrho))} + c_3 e^{-c_2 (s-\varrho-R) \log(1+ c_2 (s-\varrho-R))} \\ & + c_5 K(r) + \sup_{(t,\theta,y) \in \ T \times E \times B(0,\varrho)} I_4(s,t,\theta,y,R) + e^{\kappa \varrho}\sup_{(t,\theta) \in T \times E} I_5(t,\theta,R) ,
\end{align*}
for sufficiently large $R>0$ and $s>0$, with constants $c_1$, $c_2$, $c_3$ and $c_5$ uniform in $s> 3(R + \varrho)$ ($c_5$ also uniform in $R>3r_0 \vee R_0$). Since
\begin{align*}
& \sup_{(t,\theta,y) \in \ T \times E \times B(0,\varrho)} I_4(s,t,\theta,y,R) \\ & \ \ \ \ \ \ \ \leq \sup_{(t,\theta,z) \in \ T \times E \times B(0,\varrho+R)} \left|\frac{\bar{p}_t (s \theta - z)}{t \, \nu(s \theta)} -  e^{\kappa (\scalp{\theta}{(z)})} \, \exp\left( t \int_{|w|>h(t)}\left(e^{\kappa (\scalp{\theta}{w})} - 1\right) \, \nu(w)\, dw\right)\right|,
\end{align*}
Lemma \ref{lem:conv_Poiss} (b) yields 
$$
\limsup_{s \to \infty} \sup_{(t,\theta,y) \in \ T \times E \times B(0,\varrho)} W(s,t,\theta,y)  \leq c_5 K(r) + c_7 t_0 e^{C_{10}+\kappa \varrho} \int_{|z|>R} e^{-c_8 \big(\log(1+c_8 R) - \frac{\kappa}{c_8}\big)|z|} f(|z|) dz.
$$
We conclude by taking the limit $R \to \infty$ on both sides of the above inequality. 
\end{proof}

We now discuss the possible converse implications between the convergence \eqref{eq:res_conv} of Theorem \ref{thm:main} and our key conditions  
\eqref{ass:sjp} and \eqref{eq:add_conv}. 

\begin{prop}\label{prop:converse}
Let $\big\{P_t: t \geq 0 \big\}$ be a semigroup of probability measures determined by \eqref{eq:Phi_2} such that the densities $p_t$ exist. Then we have the following.
\begin{itemize}
\item[(a)] If there exist $E \subset \sfera$, $\kappa >0$ and $t_0 > 0$ such that for every $\theta \in E$ the measures $P_t$ has mutlidimensional exponential moments of order $\kappa \theta$ and for every $\theta \in E$ the convergence 
\begin{align} \label{eq:lim_1}
  \lim_{r \to \infty} \frac{p_t(r\theta -y)}{t \, \nu(r\theta)} = e^{-t \widetilde \psi(\kappa \theta) + \kappa (\scalp{\theta}{y})} 
\end{align}
holds uniformly in $t$ on $(0,t_0)$ and locally uniformly in $y \in \R^d$ (cf. \eqref{eq:res_conv}), then 
\begin{align} \label{eq:lim_2}
  \lim_{r \to \infty} \frac{\nu(r\theta -y)}{\nu(r\theta)} = e^{\kappa (\scalp{\theta}{y})},
\end{align}
for every $\theta \in E$ and almost every $y \in \R^d$ (cf. \eqref{eq:add_conv}). If, in adition, the convergence in \eqref{eq:lim_1} is also uniform in $\theta \in E$, then the same is true for \eqref{eq:lim_2}.

\item[(b)] If there exist $E \subset \sfera$ and $t_0 > 0$ such that for every $\theta \in E$ the convergence 
\begin{align} \label{eq:lim_3}
  \lim_{r \to \infty} \frac{p_t(r\theta -y)}{t \, \nu(r\theta)} = 1 
\end{align}
holds uniformly in $t$ on $(0,t_0)$ and locally uniformly in $y \in \R^d$, then 
\begin{align} \label{eq:lim_4}
  \lim_{r \to \infty} \frac{\nu(r\theta -y)}{\nu(r\theta)} = 1,
\end{align}
for every $\theta \in E$ and almost every $y \in \R^d$. If, in adition, the convergence in \eqref{eq:lim_3} is also uniform in $\theta \in E$, then the same is true for \eqref{eq:lim_4}.

\item[(c)] If there exists a nonincreasing function $f:(0,\infty) \to (0, \infty)$ such that $\nu(x) \asymp f(|x|)$, $x \in \R^d \setminus \left\{0\right\}$, $t_0 > 0$, and the functions $\eta_1, \eta_2 :\sfera \to (0,\infty)$ such that $ 0 < \inf_{\theta \in \sfera} \eta_1(\theta) \leq \sup_{\theta \in \sfera} \eta_1(\theta) < \infty$, $ 0 < \inf_{\theta \in \sfera} \eta_2(\theta) \leq \sup_{\theta \in \sfera} \eta_2(\theta) < \infty$ and 
\begin{align} \label{eq:lim_5}
  \lim_{r \to \infty} \frac{p_t(r \theta)}{t \, \nu(r \theta)} = \left\{
  \begin{array}{ccc}
   \eta_1(\theta) & \mbox{  if  } & t=t_0,\\
   \eta_2(\theta) & \mbox{  if  } & t=2t_0,
  \end{array}\right.
\end{align}
uniformly in $\theta$ on $\sfera$, then we have $K(r) < \infty$, for every $r \geq 1$. 
\end{itemize}
\end{prop}

\begin{proof}
We only prove the assertion (a) and (c). The proof of (b) is just a simpler version of that of (a).

\vspace{0.1cm}

\noindent
(a) First recall that $\lim_{t \to 0^{+}} \frac{p_t(x)}{t} = \nu(x)$ vaguely on $\R^d \setminus \left\{0\right\}$. By the Portmanteau theorem, however, this implies that 
\begin{align} \label{eq:portmanteau}
\lim_{t \to 0^{+}} \int_{B(y, \varepsilon)} \frac{p_t(x)}{t} dx = \int_{B(y, \varepsilon)} \nu(x) dx, \quad \mbox{for every $\varepsilon >0$ and $y \in \R^d$ such that $|y| > \varepsilon$. }
\end{align}
Also, since the convergence in \eqref{eq:lim_1} holds uniformly in $t$ on $(0,t_0)$ and locally uniformly in $y \in \R^d$, we get that for every $y \in \R^d$, $\theta \in E$ and $\delta >0$, there exists $\varepsilon_0 >0$ and $R >|y|+\varepsilon_0$ such that for every $t \in (0,t_0)$, $z \in B(y,\varepsilon_0)$ and $r \geq R$ we have
$$
\left|\frac{p_t(r\theta - z)}{t\nu(r\theta)} - e^{-t \widetilde \psi(\kappa \theta) + \kappa (\theta \cdot z)}\right| \leq \delta.
$$
In particular, for every $\varepsilon \in (0,\varepsilon_0)$, $t \in (0,t_0)$ and $r \geq R$, 
\begin{align*}
\left(-\delta |B(y,\varepsilon)| + e^{-t \widetilde \psi(\kappa \theta)} \int_{B(y,\varepsilon)}e^{ \kappa (\theta \cdot z)} dz\right) & \nu(r\theta) \leq \int_{B(y,\varepsilon)}\frac{p_t(r\theta - z)}{t} dz \\ & \leq \left(\delta |B(y,\varepsilon)| + e^{-t \widetilde \psi(\kappa \theta)} \int_{B(y,\varepsilon)}e^{ \kappa (\theta \cdot z)} dz\right) \nu(r\theta).
\end{align*}
By simple change of variables, the middle integral takes the form $\int_{B(r\theta-y,\varepsilon)}\frac{p_t(z)}{t} dz$. 
So by taking the limit $t \to 0^{+}$, thanks to \eqref{eq:portmanteau} we get 
\begin{align*}
\left(-\delta |B(y,\varepsilon)| + \int_{B(y,\varepsilon)}e^{ \kappa (\theta \cdot z)} dz\right) & \nu(r\theta) \leq \int_{B(r\theta-y,\varepsilon)}\nu(z) dz \\ & \leq \left(\delta |B(y,\varepsilon)| + \int_{B(y,\varepsilon)}e^{ \kappa (\theta \cdot z)} dz\right) \nu(r\theta).
\end{align*}
Now, by changing variables one more time, by dividing all members of this chain of inequalities by $|B(y,\varepsilon)|$ and by taking the limit $\varepsilon \to 0^{+}$, we finally get
$$
\left(-\delta + e^{\kappa (\theta \cdot y)}\right) \nu(r\theta) \leq \nu(r\theta-y) \leq \left(\delta + e^{\kappa (\theta \cdot y)}\right) \nu(r\theta), 
$$
for almost all $y \in \R^d$, with $\delta$ depending on $y$. This clearly gives that the limit in \eqref{eq:lim_2} holds for every $\theta \in E$ and almost every $y \in \R^d$. By inspection of the above argument, we also see that the uniform convergence in $\theta$ on $E$ in \eqref{eq:lim_1} implies the same for \eqref{eq:lim_2}. This completes the proof of (a).

\vspace{0.1cm}

\noindent
(c) Observe that by \eqref{eq:lim_5} there exist the constants $0< c_1, c_2 < \infty$ and $R >0$ such that for every $|x| \geq R$ we have $p_{t_0}(x) \geq c_1 f(|x|)$ and $p_{2t_0}(x) \leq c_2 f(|x|)$. Thus, by the semigroup property, we get
\begin{align*}
c_2 f(|x|) \geq p_{2t_0}(x) & = \int_{\R^d} p_{t_0}(x-y) p_{t_0}(y) dy \geq \int_{|x-y|> R \atop |y|>R} p_{t_0}(x-y) p_{t_0}(y) dy \\
                            & \geq c_1^2 \int_{|x-y|> R \atop |y|>R} f(|x-y|) f(|y|) dy, 
\end{align*}
which immediately implies that $K(R) < \infty$. Hence, by strict positivity and monotonicity properties of the profile $f$, 
from this we can also derive that $K(r) < \infty$, $r \geq 1$. 
\end{proof}
From the above proof we see that if we know that the limit $\lim_{t \to 0^{+}} \frac{p_t(x)}{t} = \nu(x)$ is pointwise on $\R^d \setminus \left\{0\right\}$, then in both parts (a) and (b) of the above proposition it is enough to assume that the convergence in \eqref{eq:lim_1} and \eqref{eq:lim_3} is only pointwise in $y \in \R^d$.

\bigskip

\section{The case of compound Poisson convolution semigroups}

In this section we give the proof of Theorem \ref{thm:main_second} which is devoted to finite L\'evy measures. 
The argument is a modification of that from the previous sections. First recall that if $\nu(\R^d \setminus \left\{0\right\}) < \infty$, then $\Psi(\infty) < \infty$. This together with Fatou's lemma leads to the following two direct corollaries from Lemmas \ref{lm:useful} and \ref{lm:useful_2}. Recall also that
$$
\widetilde{p}_t(x) :=  e^{-t|\nu|} \sum_{n=1}^\infty \frac{t^n\nu^{n*}(x)}{n!}, \quad  t>0,
$$
stand for the densities of the absolutly continuous components of the measures $\widetilde P_t$.

\begin{corollary}\label{cor:useful} 
Let $\nu(\R^d \setminus \left\{0\right\}) < \infty$ and let the assumption \textbf{(B)} holds. Then for every fixed $r_0>0$ we have the following. 
\begin{itemize}
\item[(a)]
There are constants $C_4=C_4(r_0)$ and $C_5=C_5(r_0)$ such that
\begin{align*} %\label{eq:G_1}
\int_{|x -y| > r_0 } f(|x-y|) \nu(y) dy \leq C_4 \Psi(\infty) f(|x|), \quad |x| \geq 2r_0.
\end{align*}
\item[(b)]
There exists a constant $C_6=C_6(r_0) \geq 1$ such that
$$
f(s-r_0) \leq C_6 f(s), \quad s \geq 3r_0.
$$
\item[(c)] For every numbers $C_7, C_8 >0$ there exists a constant $C_9:=C_9(r_0)>0$ such that
$$
e^{-C_7 s \log(1+ C_8 s)} \leq C_9 f(s), \quad s \geq r_0.
$$ 
\item[(d)] There is a constant $C_{10}=C_{10}(r_0)$ such that
\begin{align*}
\int_{|x-y|>r_0} f(|y-x|) \nu^{\ast n}(y)\, dy \leq \left(C_{10} \Psi(\infty)\right)^n f(|x|) ,\quad |x|\geq 3r_0, \ \ n \in \N.
\end{align*}
\item[(e)] There exists $C_{11}=C_{11}(r_0)$ such that for every $n \in \N$ we have
  \begin{align*}
    \nu^{n*}(x) \leq C_{11}^n \left[\Psi(\infty)\right]^{n-1} f(|x|), \quad |x|>3r_0.
  \end{align*}
\item[(f)] There exists $C_{12}=C_{12}(r_0)$ such that we have
  \begin{align*}
    \widetilde{p}_t(x) \leq C_{12} \, t \, f(|x|), \quad |x|>3r_0, \ t \in (0,t_0], 
  \end{align*}
	with $t_0:= 1/\Psi(1/r_0)$.
\end{itemize}
\end{corollary}

\begin{corollary}\label{cor:useful_2} 
Let $\nu(\R^d \setminus \left\{0\right\}) < \infty$ and let the assumptions \textbf{(B)} and \textbf{(C)} hold with some $E \subset \sfera$ and $\kappa \geq 0$. There exists a constant $C_{17} >0$ such that
$$
\int_{\R^d} e^{\kappa (\scalp{\theta}{z})} \nu^{n*}(z) \, dz \leq (C_0/C_1) \big(C_{17} \Psi(\infty) \big)^n, \quad  \ \theta \in E, \ n \in \N,
$$ 
and
$$
\lim_{R \to \infty} \sup_{\theta \in  E} \int_{|z| > R} e^{\kappa (\scalp{\theta}{z})} \nu^{n*}(z) \, dz = 0, \quad n \in \N.
$$
\end{corollary}

By following the lines of the proofs of Lemmas \ref{lem:conv} and \ref{lem:conv_Poiss}, based on Lemmas \ref{lm:useful} and \ref{lm:useful_2} replaced with Corollaries \ref{cor:useful} and \ref{cor:useful_2} formulated above, we obtain the following results on the convergence of convolutions of the finite L\'evy densities and the corresponding densities $\widetilde{p}_t$. This can be done by direct inspection and, therefore, the proofs are omitted.

\begin{lemma} \label{lem:conv_finite}
Let $\nu(\R^d \setminus \left\{0\right\}) < \infty$ and let the assumptions \textbf{(B)} and \textbf{(C)} hold with some $E \subset \sfera$ and $\kappa \geq 0$. Then we have the following. 
\begin{itemize}
\item[(a)]
For every $n \in \N$, $\theta \in E$ and $y \in \R^d$
\begin{equation}\label{eq:conv_limit_0_finite}
	\lim_{s \to \infty} \frac{\nu^{n*}(s \theta-y)}{\nu(s \theta)} = 
	e^{\kappa (\scalp{\theta}{y})} n \left(\int e^{\kappa (\scalp{\theta}{z})}\nu(z)\, dz\right)^{n-1}.
\end{equation}
\item[(b)] If the convergence in \textbf{(C)} is uniform in $(\theta,y)$ on each rectangle $E \times D$, for every compact set $D \subset \R^d$, then for any $n \in \N$ the convergence \eqref{eq:conv_limit_0_finite} is uniform in $(\theta,y)$ on each rectangle $E \times B(0,\varrho)$, $\varrho >0$.
\end{itemize}
\end{lemma}

\begin{lemma} \label{lem:conv_Poiss_finite}
Let $\nu(\R^d \setminus \left\{0\right\}) < \infty$ and let the assumptions \textbf{(B)}-\textbf{(C)} hold with some $E \subset \sfera$ and $\kappa \geq 0$. Moreover, let $t_0>0$ be arbitrary.
\begin{itemize}
\item[(a)]
For every $t \in (0,t_0]$, $\theta \in E$ and $y \in \R^d$ one has
\begin{equation}\label{eq:conv_Poisson_finite}
	\lim_{s \to \infty} \frac{\widetilde{p}_t(s \theta-y)}{t \, \nu(s \theta)} = 
	\exp\left(\kappa (\scalp{\theta}{y}) \right) \, \exp\left( t \int \left(e^{\kappa (\scalp{\theta}{z})} - 1\right) \, \nu(z)\, dz\right).
\end{equation} 
\item[(b)] If the convergence in \textbf{(C)} is uniform in $(\theta,y)$ on each rectangle $E \times D$, for every compact set $D \subset \R^d$, then the convergence in \eqref{eq:conv_Poisson_finite} is uniform in $(t,\theta,y)$ on each cuboid $(0,t_0] \times E \times B(0,\varrho)$, $\varrho >0$.
\end{itemize}
\end{lemma}

We are now in position to give the proof of our second main Theorem \ref{thm:main_second}.

\begin{proof}[Proof of Theorem \ref{thm:main_second}]
The proof is in fact a version of that of Theorem \ref{thm:main}. Theorefore, we mainly focus on pointing out the crucial differences and omit the details. Let the assumptions \textbf{(A)} and \textbf{(B)} hold and let $\nu(\R^d \setminus \left\{0\right\}) < \infty$. Moreover, consider arbitrary $b \in \R^d$ and suppose that the assumption \textbf{(C)} hold with some $E \subset \sfera$ and $\kappa \geq 0$. Let $r_0 = 1$ and $R_0 \geq r_0=1$ be the radius provided by Lemma \ref{lem:lambda_est} (a) and let $t_0>0$ be fixed. 

We first consider the case $\inf_{|\xi|=1} \xi \cdot A \xi > 0$ in the assumption \textbf{(A)}. 

\vspace{0.1cm}

\noindent
(a) It suffices to show that for every $t \in (0,t_0]$, $\theta \in E$ and $y \in \R^d$,
\begin{align} \label{eq:res_conv_initial_finite}
\lim_{r \to \infty} \frac{p_t(r\theta -y + t \widetilde b )}{t \, \nu(r\theta)} =  e^{\kappa (\scalp{\theta}{y})} \exp\left(t\int \left( e^{\kappa (\scalp{\theta}{z})} - 1 \right)\nu(z) dz\right) \, \int_{\R^d} e^{\kappa (\scalp{\theta}{z})} g_t(z) dz.
\end{align}
If this is true, then by substituting $y = w + t \widetilde b$ (recall that $\widetilde b = b - \int_{|y|<1} y \nu(y) dy$), the assertion (a) of the theorem also holds.

The convergence in \eqref{eq:res_conv_initial_finite} can be justified by following the estimates in the proof of Theorem \ref{thm:main} (a).
First of all, note that the counterpart of \eqref{eq:W_est} can be established by applying the decomposition formula \eqref{eq:basic_conv_3} instead of \eqref{eq:basic_conv_1}, with $\lambda_t(z)$ replaced by $g_t(z)$. All the members on the right hand side of this estimate can  be then effectively estimated by using Corollaries \ref{cor:useful}-\ref{cor:useful_2} and Lemma  \ref{lem:lambda_est} (a) instead of Lemmas \ref{lm:useful}-\ref{lm:useful_2} and \ref{lem:lambda_est} (b). Also, the convergence of the corresponding member $I_4$ to zero follows directly from Lemma \ref{lem:conv_Poiss_finite}. 

\vspace{0.1cm}

\noindent
(b) If the convergence in \textbf{(C)} is uniform in $(\theta,y)$ on each rectangle $E \times D$, for every compact set $D \subset \R^d$, then by the fact that $|t \widetilde b|$ is uniformly bounded in $t \in (0,t_0]$, it is enough to prove that \eqref{eq:res_conv_initial_finite} holds uniformly in $(t,\theta,y)$ on each cuboid $(0,t_0] \times E \times B(0,\varrho)$, $\varrho >0$. However, this can be done
exactly in the same way as in the proof of Theorem \ref{thm:main} (b), by justifying that all the estimates established in part (a) and the convergence of the countepart of $I_4$ to zero are also uniform. 
\vspace{0.1cm}

\noindent

If $A \equiv 0$ in assumption \textbf{(A)}, then the assertions of the theorem follows directly from Lemma \ref{lem:conv_Poiss_finite}. 
\end{proof}

\section{Discussion of specific classes of semigroups} \label{sec:Examples}
\subsection{Stable semigroups (possibly with drift and Gaussian part)}

First we consider the well known example of stable semigropus, e.g., semigroups generated by the L\'evy measures

\begin{equation}\label{stablenu}
	\nu(x)\, dx = |x|^{-\alpha-d}g(x/|x|)\, dx,\quad x\in\Rd\setminus\{ 0\},
\end{equation}
where $\alpha\in (0,2)$ and $g:\sfera \to [0,\infty)$ is such that $0 \leq g(\theta) \leq c$, $\theta\in\sfera$ for some
positive constant $c$ and 
\begin{equation}\label{nondegenerete}
  \int_{\sfera} g(\theta)\, d\theta\, >0.
\end{equation}

%\begin{equation}\label{nondegenerete}
%  \inf_{u\in\sfera} \int_{\sfera} |\scalp{u}{\theta}|^\alpha g(\theta)\, d\theta\, >0.
%\end{equation}

It is straightforward to verify that for $f(s)=s^{-\alpha-d}$ we have $K(r)\asymp r^{-\alpha} \to 0$ as $r \to \infty$ and, by \eqref{nondegenerete}, the condition \eqref{ass:low_reg} holds as well. Thus the assumption \textbf{(B)} is satisfied. We also observe that in this case for $\theta \in \sfera$ and $y\in\Rd$ we have
\begin{eqnarray*}
	\lim_{r \to \infty} \frac{\nu(r \theta-y)}{\nu(r\theta)} 
	&  =  & \lim_{r \to \infty} \frac{|r\theta - y|^{-\alpha-d}g((r\theta-y)/|r\theta-y|)}{r^{-\alpha-d}g(\theta)} \\
	&  =  & \lim_{r \to \infty} \frac{|\theta - y/r|^{-\alpha-d}g((\theta-y/r)/|\theta-y/r|)}{g(\theta)} = 1,
\end{eqnarray*}
uniformly in $y$ on $B(0,\rho)$ for every $\rho>0$, provided that $g$ is positive and continuous at $\theta$. 
If $g$ is uniformly continuous and bounded from below by a positive constant
on $E \subset\sfera$ then the convergence is uniform in $(\theta,y)$ on $E \times B(0,\rho)$ for
every $\rho>0$. This clearly gives \textbf{(C)}. Moreover, the assumption \textbf{(D)} holds with $\Psi_{-}(1/t) \asymp t^{-1/\alpha}$ and $T=(0,t_0)$, for
every $t_0>0$ since we have $\Re \Phi(\xi) \asymp |\xi|^{\alpha}$. Hence using Theorem \ref{thm:main} we obtain the following result.

\begin{theorem}\label{stable_ex}
  If $A=0$ or $\inf_{|\xi|=1} \xi \cdot A \xi > 0$, $b\in\Rd$ and the L\'evy measure $\nu$ is given by \eqref{stablenu}, where $0 \leq g(\theta) \leq c$,  $\theta\in\sfera$, for some constant
	$c>0$, and $g$ satisfies \eqref{nondegenerete}, 
	then for every $t>0$ there exists a  density $p_t$ and for 
	every $\theta\in\sfera$ such that
	$g$ is positive and continuous at $\theta$ we have
	$$
	  \lim_{r \to \infty} \frac{p_t(r\theta -y)}{t \, r^{-\alpha-d}} = g(\theta),
	$$
	uniformly in $(t,y)$ on each rectangle $(0,t_0)\times B(0,\rho)$, $t_0>0$, $\rho>0$.
	If $g$ is uniformly continuous on $E \subset\sfera$ and $g(\theta)\geq c_1>0$
	for $\theta \in E$ then the convergence is uniform
	in $(t,\theta,y)$ on $(0,t_0) \times E \times B(0,\rho)$ for every
	$t_0>0$ and $\rho>0$.
\end{theorem}

Note that if we have $A=0$ and
$$
  b = \left\{
  \begin{array}{ccc}
    \int_{|y|<1} y \, \nu(dy)   & \mbox{  if  } & \alpha < 1,\\
    0                           & \mbox{  if  } & \alpha=1, \\
    - \int_{1\leq |y|} y \, \nu(dy) & \mbox{  if  } & \alpha > 1,
  \end{array}\right.
$$
and for $\alpha=1$  additionally
$$
  \int_{\sfera} \theta\,g(\theta) d\theta = 0,
$$
then we obtain a strictly stable semigroup (see \cite[Th. 14.7]{Sato}). We note that for strictly stable semigroups of measures similar result was obtained by J Dziuba\'nski in \cite{D91} under stronger assumption that $g$ is symmetric and continuous on $\sfera$. The main novelty of our present result
for stable semigroups in $\R^d$ is that it does not require any symmetry assumptions (we can treat even highly asymmetric spherical densities $g$) and that it is local on the sphere $\sfera$, i.e. we obtain the asymptotics in generalized cones $\Gamma_E$ for arbitrary subsets $E \subset \sfera$, provided $g$ is continuous and separated from zero on $E$.

On the other hand, our present results do not apply to those $\theta \in \sfera$ for which $g(\theta) = 0$. This case is much more difficult and requires essential modifications in our present framework. Its systematic study is a subject of our ongoing project. 
  
For better illustration, we propose now to consider a particular example of stable L\'evy measure and the corresponding heat kernel. 

\begin{example} \label{ex:ex1} {\rm Let $\nu$ be a stable density on $\R^2$ given by \eqref{stablenu} with
\begin{displaymath}
  g(\theta) = g((\theta_1,\theta_2)) = \left\{
	\begin{array}{ccc}
	  1 & \text{for} & \theta_1 \theta_2\geq 0, \\
	  2 & \text{for} & \theta_1 \theta_2 < 0.
  \end{array}
  \right.
\end{displaymath} 
It follows from Theorem \ref{stable_ex} that for such $\nu$, $A=0$ and $b=0$ we have
\begin{equation}\label{ex_stable2}
  \lim_{r\to\infty} \frac{p_t(r\theta-y)}{tr^{-\alpha-2}} = \left\{
	\begin{array}{ccc}
	  1 & \text{for} & \theta_1 \theta_2 > 0, \\
	  2 & \text{for} & \theta_1 \theta_2 < 0.
  \end{array}
  \right.
\end{equation} 

We note that the above convergence is uniform on every cube
$$(0,t_0)\times \{\theta \in \sfera :\: \theta_1\theta_2 > \delta\}\times B(0,\rho)$$ and 
$$ (0,t_0)\times \{\theta \in \sfera :\: \theta_1\theta_2<-\delta\}\times B(0,\rho),$$ for all $t_0>0,\delta\in (0,1/2)$, $\rho>0$.
This, however, yields that there exists $R>0$ such that
\begin{displaymath}
	p_t(r\theta)\leq \frac{5}{4} t r^{-\alpha-2},\quad r>R, \ \ \theta_1\theta_2>\delta,
\end{displaymath}
and
\begin{displaymath}
	p_t(r\theta)\geq \frac{3}{2} t r^{-\alpha-2}, \quad r>R, \ \ \theta_1\theta_2<-\delta.
\end{displaymath}
Hence the uniform continuity can not hold for any cube above with $\delta=0$, since it contradicts the continuity of $p_t$
(note that $p_t$ is smooth function for every $t>0$ since $\Re \Phi(\xi) \asymp |\xi|^{\alpha}$). 
We also do not know what happens for $\theta \in \sfera$ with $\theta_1 \theta_2 = 0$. 
%continuuity of $p$ ..... convergence is not uniform on $\sfera\setminus\{\theta\in\sfera:\: \theta_1\theta_2=0\}$
%where $R_2$ does not depend on $\theta$ since the convergence in \eqref{ex_stable2} is
%uniform on $\sfera\setminus\{\theta\in\sfera:\: \theta_1\theta_2=0\}$. This is a contradiction
%since $p_t$ is continuous on whole $\R^2$. Therefore the limit 
%$\lim_{r\to\infty}\frac{p_t(r\theta)}{tr^{-\alpha-2}}$ differs from $g(\theta)$ or
%does not exist.
}
\end{example}

We can also slightly modify the stable examples considering densities such that 
\begin{displaymath}
	  \nu(x)=g(x/|x|) |x|^{-\alpha-d}[\log(1+|x|^{-\kappa})]^{-\beta},
\end{displaymath}
where $\alpha\in(0,2]$, $\kappa>0, \alpha>\kappa\beta>\alpha-2$ and $\beta>1$ if $\alpha=2$.
In this case we have $\Re \Phi(\xi)\asymp |\xi|^{\alpha}[\log(1+|\xi|^\kappa)]^{-\beta}$, $\Psi_{-}(1/t)\asymp t^{-1/\alpha} \left[\log\left(1+\frac{1}{t}\right)\right]^{\frac{\beta}{\alpha}}$
if $\alpha\in (0,2)$, $t<t_0$ and $\Re \Phi(\xi)\asymp |\xi|^{2}[\log(1+|\xi|^{\kappa\beta/(\beta-1)})]^{1-\beta}$, $\Psi_{-}(1/t)\asymp  t^{-1/2} \left[\log\left(1+\frac{1}{t}\right)\right]^{\frac{\beta-1}{2}}$
if $\alpha = 2$, $t<t_0$, for every $t_0>0$ (see \cite[Thm. 4]{KSz1}). Then the results analogous to Theorem \ref{stable_ex} also hold with $g$ satisfying the same conditions. We omit the straightforward verification of assumptions of Theorem \ref{thm:main}. 

\vspace{1cm}

For the next examples we need the following Lemma.
We consider here a class of dominating profiles (majorants) for L\'evy measures with polynomial, stretched-exponential, 
exponential and super-exponential decay at infinity and give a full characterization of the condition \eqref{ass:sjp} in
assumption \textbf{(B)} for this class.

\begin{lemma}\label{see_charB}
   Let $m\geq 0$, $\beta \geq 0$, $\delta\geq 0$, 
	 % such that $\nu(x)\asymp f(|x|)$, 
	 %for $|x|\geq 1$, where
	 and
	 \begin{displaymath}
		 f(s) = e^{-ms^{\beta}} s^{-\delta},\quad s>1.	
	 \end{displaymath}
	 Then the condition \eqref{ass:sjp} in \textbf{(B)} holds exactly in the following three disjoint cases (if and only if)
	 \begin{itemize}
     \item[(a)] $m=0$ and $\delta>d$,
     \item[(b)] $m>0$, $\beta \in (0,1)$ and $\delta \geq 0$,
     \item[(c)] $m>0$, $\beta =1$ and $\delta > (d+1)/2$. 
   \end{itemize}
\end{lemma}  
\begin{proof} The proof that (a) yields \eqref{ass:sjp} is straightforward and we omit the details. For the proof that (b) 
	implies \eqref{ass:sjp} we refer to the proof of \cite[Cor. 4.2]{KL17} and the proof for (c) is
	a simple modification of the corresponding part of proof in \cite[Prop. 2]{KSz2}. The converse implications follow
	directly from \cite[Prop. 2]{KSz2}.
\end{proof}

\subsection{Relativistic stable semigroups}

We consider now an important class of evolution semigroups corresponding to the so-called relativistic stable operators $L=-(m^{2/\alpha}-\Delta)^{\alpha/2}+m$, $\alpha \in (0,2)$, $m>0$ (see e.g. \cite{CMS, KS2006, R2002, S11, KL, KSz2, ChKimSon}). The operator $H_0 = \sqrt{m^2-\Delta}+m$ (i.e. $\alpha = 1$) is known to describe the kinetic energy of a free quasi-relativistic particle and is one of the central objects of the modern investigations in PDEs and mathematical physics (see e.g. \cite{C, FJL, LS, HL} and references therein).
	
Let $A=0$, $b=0$, $\alpha \in (0,2)$, $m>0$, and
  \begin{equation}\label{relstabnu}
    \nu(x)=\frac{c_{d,\alpha}}{|x|^{d+\alpha}}e^{-m^{1/\alpha}|x|}\varphi(m^{1/\alpha}|x|),
  \end{equation}
	where
	$$
	  \varphi(\xi)=\int_0^\infty e^{-v}v^p (\xi+v/2)^p\, dv,\quad \xi\geq 0, \ \  p=\frac{d+\alpha-1}{2},
	$$
	and $c_{d,\alpha}=\Gamma((d+\alpha)/2)/(\pi^{d/2}2^{-\alpha/2}|\Gamma(-\alpha/2)|\varphi(0)).$ We have
	\begin{displaymath}
		\nu(x) \asymp   e^{-m^{1/\alpha}|x|} |x|^{-\frac{d+\alpha+1}{2}}, \quad |x|>1,
	\end{displaymath}
	and
	\begin{displaymath}
		\nu(x) \asymp   |x|^{-d-\alpha}, \quad |x|\leq 1,
	\end{displaymath}
	and therefore it follows from Lemma \ref{see_charB} that \textbf{(B)} holds for $\nu$
	and 
	$$
	  f(s) = e^{-m^{1/\alpha}} s^{-d-\alpha}\indyk{(0,1]}(s) +  e^{-m^{1/\alpha}s} s^{-\frac{d+\alpha+1}{2}}\indyk{(1,\infty)}(s).
	$$ 
	Moreoever,
	\begin{eqnarray*}
	  \lim_{r\to\infty}\frac{\varphi(m^{1/\alpha}|r\theta-y|)}{\varphi(m^{1/\alpha}r)} 
		&  =   & \lim_{r\to\infty}\frac{\int_0^\infty e^{-v}v^p (m^{1/\alpha}|r\theta-y|+v/2)^p\, dv}{\int_0^\infty e^{-v}v^p (m^{1/\alpha}r+v/2)^p\, dv} \\
		&  =   & \lim_{r\to\infty}\frac{\int_0^\infty e^{-v}v^p (m^{1/\alpha}|\theta-y/r|+\frac{v}{2r})^p\, dv}{\int_0^\infty e^{-v}v^p (m^{1/\alpha}+\frac{v}{2r})^p\, dv} = 1,
	\end{eqnarray*}
	since 
	$$\lim_{r\to\infty} \int_0^\infty e^{-v}v^p (m^{1/\alpha}+\frac{v}{2r})^p\, dv = \lim_{r\to\infty} \int_0^\infty e^{-v}v^p (m^{1/\alpha}|\theta-y/r|+\frac{v}{2r})^p\, dv = m^{p/\alpha} \int_0^\infty e^{-v}v^p \, dv, $$
	which follows from bounded convergence. Now we easily get
	$$
	  \lim_{r\to\infty} \frac{\nu(r\theta-y)}{\nu(r\theta)} = e^{m^{1/\alpha}(\scalp{\theta}{y})},
	$$
	uniformly in $(\theta, y)$ on each rectangle $\sfera \times B(0,\rho)$, for every $\rho>0$. Thus \textbf{(C)} holds with $E = \sfera$.
	
	The underlying
	semigroup has the characteristic exponent of the form (recall that $A = 0$ and $b=0$)
	\begin{displaymath}
		\psi(\xi) = \Phi(\xi) = (m^{2/\alpha} + |\xi|^2)^{\alpha/2}-m, \quad \xi\in\Rd.
	\end{displaymath}
	Since
	$\Phi(\xi)\asymp |\xi|^2 \wedge|\xi|^{\alpha}$ and $1/\Psi_{-}(1/t)\asymp t^{1/2} \wedge t^{1/\alpha} $,
	the assumption \textbf{(D)} is also satisfied for every $T=(0,t_0)$, with $t_0>0$, and we get the following result.
	
 \begin{theorem}
   If $A=0$, $b=0$ and the L\'evy measure $\nu$ is given by \eqref{relstabnu}, 
	 then there exist densities $p_t$ and for every $\theta\in\sfera$, we have
   \begin{align}
     \lim_{r \to \infty} \frac{p_t(r\theta -y)}{t \, \nu(r\theta)} = e^{mt + m^{1/\alpha} (\scalp{\theta}{y})},
   \end{align}
   uniformly in $(t,\theta,y)$ on each set $(0,t_0) \times \sfera \times B(0,\rho)$, $t_0>0$, $\rho>0$.
 \end{theorem}
 \begin{proof}
   We have already verified the assumptions of Theorem 1 above. We need only to check that 
	$\widetilde{\psi}(m^{1/\alpha}\theta)=-m$. Using \cite[Th. 25.17]{Sato} and Lemma \ref{lm:useful_2}, we get 
	$e^{-t\widetilde{\psi}(\xi)} = \int_{\Rd} e^{\scalp{\xi}{y}} p_t(y)\, dy$, for every $t>0$ and $\xi \in \R^d$ such hat $|\xi|\leq m^{1/\alpha}$. Furthermore, we have $$p_t(y) = e^{mt} \int_0^\infty \left(\frac{1}{4\pi s}\right)^{d/2}e^{\frac{-|y|^2}{4s}}e^{-m^{2/\alpha}s}\eta(t,s)\, ds, $$ where
	$\eta(t,s)$ is the transition density of an $\alpha/2$-stable subordinator such that 
$\int_0^\infty e^{-\lambda s} \eta(t,s)\, ds = e^{-t\lambda^{\alpha/2}}$, $\lambda\geq 0$ (see e.g. \cite{R2002}).
Now we easily get
\begin{eqnarray*}
  e^{-t\widetilde{\psi}(\xi)}
	& = & e^{mt} \int_0^\infty \left(\frac{1}{4\pi s}\right)^{d/2} 
	             \left( \int_{\Rd}e^{\scalp{\xi}{y}-\frac{|y|^2}{4s}}\, dy \right) 
							 e^{-m^{2/\alpha}s}\eta(t,s) \, ds \\
	& = & e^{mt} \int_0^\infty e^{-(m^{2/\alpha}-|\xi|^2)s}\eta(t,s) \, ds \\
	& = & e^{mt} e^{-t(m^{2/\alpha}-|\xi|^2)^{\alpha/2}}, \quad |\xi|\leq m^{1/\alpha}.
\end{eqnarray*}
In particular, $e^{-t\widetilde{\psi}(m^{1/\alpha}\theta)} = e^{mt}$.
	
 \end{proof}

\subsection{Semigroups with stretched-exponentially localized L\'evy measures}  \label{sec:subexp}
Let
\begin{equation}\label{Lm:subexp}
  \nu(x) = g(x/|x|)f(|x|),
\end{equation}
where
\begin{align} \label{def:off}
  f(s) = \indyk{[0,1]}(s) \cdot \eta(s)  + c_0 \ \indyk{(1,\infty)}(s) \cdot e^{-m s^{\beta}} s^{-\delta}, \quad s \geq 0.
\end{align}
We assume here that $\beta\in (0,1)$, $\delta\geq 0$, and $\eta: [0,1] \to (0,\infty]$ is a nonincreasing function such that $c_0e^{-m}\leq \eta(1)<\infty$, $\eta(0)=\infty$, 
and there exists $c_1 \geq 1$ satisfying $\eta(r) \leq c_1 \eta(2r)$, $r \in (0,1/2)$. Also, let $g:\: \sfera\to [0,\infty)$ be a function
such that $ 0 \leq g(\theta) \leq c_2$, for all $\theta\in\sfera$ and some positive constant $c_2$
and $g$ satisfies the nondegeneracy condition \eqref{nondegenerete}.

Obviously Lemma \ref{see_charB} yields that such L\'evy measures satisfy \textbf{(B)} with given profile $f$. We have
\begin{displaymath}
	\lim_{r\to\infty}\frac{\nu(|r\theta-y|)}{\nu(r\theta)} = 1,
\end{displaymath}  
for every $\theta\in\sfera$ and $y\in\Rd$ provided $g$ is positive and 
continuous at $\theta$ and the convergence
is uniform at $(\theta,y)\in E \times B(0,\rho)$ for every $\rho>0$ provided 
$g$ is uniformly continuous 
and bounded from below by a positive constant 
on $E\subset\sfera$.

The assumption \textbf{(D)} in fact depends only on the singularity of the function $\eta$ at zero (cf. Remark \ref{rem:assumptions} (e)).
For instance, it holds with $T = (0,t_0)$, for any $t_0>0$, if only there exist constants $M_1,M_2>0$ and 
$d<\beta_1\leq \beta_2 < d+2$ such that
$M_1(R/r)^{\beta_1}\leq \eta(r)/\eta(R) \leq M_2(R/r)^{\beta_2}$, for all $1\geq R \geq r >0$ (the proof
of this fact follows easily from a slight modification of \cite[Lem. 4.5]{S16}). 

We obtain the following theorem.

\begin{theorem}
   Let $A=0$ or $\inf_{|\xi|=1} \xi \cdot A \xi > 0$, $b\in\Rd$ and let the L\'evy measure 
	 $\nu$ be given by \eqref{Lm:subexp}, with $f$ and $g$ specified above. Assume, in addition, that $\eta$ is such 
	 that \textbf{(D)} holds on some bounded set $T\subset (0,\infty)$,
	 and that $g$ satisfies \eqref{nondegenerete}. 
	 Then for every $t \in T$ there exist densities $p_t$ and for every $\theta\in\sfera$ such that $g$ 
	 is positive and continuous at $\theta$ we have
   \begin{align}
     \lim_{r \to \infty} \frac{p_t(r\theta -y)}{t \, \nu(r\theta)} = 1,
   \end{align}
   uniformly in $(t,y)$ on each rectangle $T \times B(0,\rho)$, $\rho>0$.
	If $g$ is uniformly continuous on $E \subset\sfera$ 
	and $g(\theta)\geq c_2>0$, 
	$\theta\in E$, then the convergence is uniform
	in $(t,\theta,y)$ on $T \times E \times B(0,\rho)$ for every
	$\rho>0$.
 \end{theorem}

\subsection{Semigroups with exponentially localized L\'evy measures}
Let
\begin{equation}\label{Lm:exp}
  \nu(x) = g(x/|x|)f(|x|),
\end{equation}
where
\begin{align} \label{def:off}
  f(s) = \indyk{[0,1]}(s) \cdot \eta(s)  + c_0 \ \indyk{(1,\infty)}(s) \cdot e^{-m s} s^{-\delta}, \quad s \geq 0.
\end{align}
We assume here that $m>0$, $\delta> \frac{d+1}{2}$ and
$\eta: [0,1] \to (0,\infty]$ and $g:\sfera \to [0,\infty)$ satisfy the same assumptions as in Section \ref{sec:subexp} above.

As before, Lemma \ref{see_charB} yields that such L\'evy measures satisfy \textbf{(B)}. We have
\begin{displaymath}
	\lim_{r\to\infty}\frac{\nu(|r\theta-y|)}{\nu(r\theta)} = e^{m\scalp{\theta}{y}},
\end{displaymath}  
for every $\theta\in\sfera$ and $y\in\Rd$ provided $g$ 
is positive and 
is continuous at $\theta$ and the convergence
is uniform at $(\theta,y)\in E \times B(0,\rho)$ for every $\rho>0$ provided 
$g$ is uniformly continuous 
and separated from zero
on $E \subset\sfera$.

We have the following theorem.

\begin{theorem}
   Let $A=0$ or $\inf_{|\xi|=1} \xi \cdot A \xi > 0$, $b\in\Rd$ and let the L\'evy measure 
	 $\nu$ be given by \eqref{Lm:exp}, with $f$ and $g$ specified above. Assume, in addition, that $\eta$ is such 
	 that \textbf{(D)} holds on some bounded set $T\subset (0,\infty)$,
	 and that $g$ satisfies \eqref{nondegenerete}.
	 Then for every $t \in T$ there exist densities $p_t$ and for every $\theta\in\sfera$ such that $g$ 
	 is positive and continuous at $\theta$ we have
   \begin{align}
     \lim_{r \to \infty} \frac{p_t(r\theta -y)}{t \, \nu(r\theta)} = e^{-t\tilde{\psi}(m\theta)+m\scalp{\theta}{y}},
   \end{align}
	 where $\widetilde \psi$ is given by \eqref{eq:hyp_exp}.
  The convergence is uniform in $(t,y)$ on each rectangle $T \times B(0,\rho)$, $\rho>0$.
	If $g$ is uniformly continuous on $E \subset\sfera$ and $g(\theta)\geq c_1>0$, 
	$\theta\in E$, then the convergence is uniform
	in $(t,\theta,y)$ on $T \times E \times B(0,\rho)$ for every
	$\rho>0$.
 \end{theorem}

We end this section by discussing the following example of semigroup with isotropic exponentially 
localized L\'evy measure. It shows that the condition \eqref{eq:add_conv} of \textbf{(C)} usually does not imply \eqref{ass:sjp} in \textbf{(B)} and the existence of the exponential moments of $p_t$. In particular, even if \eqref{eq:add_conv} is true uniformly, the convergence in Theorem \ref{thm:main} may not hold without control given by \eqref{ass:sjp}.

\begin{example} \label{ex:ex2}
{\rm
Let $b \equiv 0$, $A \equiv 0$, and let $\nu$ be an isotropic L\'evy measure as in \eqref{Lm:exp}-\eqref{def:off} with $m>0$, $\delta \in  (0,\frac{d+1}{2}]$, and $g \equiv c$, for some $c >0$ \big(i.e. $\nu(x) = c f(|x|)$, $x \in \R^d \setminus \left\{0\right\}$\big). Also, let $\eta$ be such that $\int e^{-t_0 \Phi(\xi)} |\xi| d\xi < \infty$ for some $t_0>0$ \big(e.g. $\eta(s) \asymp s^{-d-\alpha}$, $s \in (0,1)$, with some $\alpha \in [0,2)$\big). By Lemma \ref{see_charB}, we have $K(r) = \infty$, for $r \geq 1$, i.e. the condition \eqref{ass:sjp} in \textbf{(B)} fails to hold. On the other hand, we can easily verify that for every compact set $D \subset \R^d$
\begin{displaymath}
	\lim_{r\to\infty}\frac{\nu(r\theta-y)}{\nu(r\theta)} = \lim_{r\to\infty}\frac{f(|r\theta-y|)}{f(r)} = e^{m(\scalp{\theta}{y})},
\end{displaymath}  
uniformly in $(\theta,y) \in \sfera \times D$. Observe, however, that $p_t$ has not exponential moments of order $m \theta$ finite:
$$
\int_{|y|>1} e^{m(\scalp{\theta}{y})} \nu(y) dy = c \int_{|y|>1} e^{-m(|y|-y_1)} |y|^{-\delta} dy = \infty,
$$
since $\delta \in  (0,\frac{d+1}{2}]$. The asymptotic property as in Theorem \ref{thm:main} also does not hold. Indeed, if we suppose that there are constants $0< c_1, c_2 < \infty$ and $t_0>0$ such that $p_{t_0}(r)/f(r) \to c_1$ and $p_{2t_0}(r)/f(r) \to c_2$ as $r \to \infty$, then there is $R>0$ such that $p_{t_0}(r) \asymp p_{2t_0}(r) \asymp f(r)$, for $r \geq R$. But this would imply $K(1)< \infty$ (cf. the proof of Proposition \ref{prop:converse} (c)), which gives a contradiction. 

This example covers many interesting processes including gamma-variance (geometric $2$-stable) process or some isotropic Lamperti transformations of stable processes.  
}
\end{example}

\bigskip

\textbf{Acknowledgements.} We thank Mateusz Kwa\'snicki for valuable discussions on the problem investigated in this paper. We also thank the Alexander von Humboldt Foundation for founding our research stays at Institut f\"ur Mathematische Stochastik, TU Dresden, Germany, during which parts of this paper have been written. We are very grateful to our host, Prof. Ren\'e L. Schilling, for his kind hospitality and numerous discussions during these stays.

\end{document}